\documentclass{amsart}

\usepackage{comment}

\usepackage[english]{babel}
\usepackage{amsmath,amssymb,amsbsy,amsthm,amsfonts,enumerate,array,color,lscape,fancyhdr,layout,pst-all}
\usepackage[pdftex]{hyperref}
\usepackage[all]{xy}
\usepackage[utf8]{inputenc}
\usepackage{hyphenat}
\usepackage{xcolor}

\usepackage{stmaryrd}
\usepackage{mathrsfs}
\usepackage{blkarray}
\usepackage{tikz}
 \usetikzlibrary{matrix,decorations.pathreplacing, calc, positioning}
 \pgfkeys{tikz/mymatrixenv/.style={decoration=brace,every left delimiter/.style=    {xshift=3pt},every right delimiter/.style={xshift=-3pt}}}
 \pgfkeys{tikz/mymatrix/.style={matrix of math nodes,left delimiter=[,right delimiter=    {]},inner sep=2pt,column sep=1em,row sep=0.5em,nodes={inner sep=0pt}}}
 \pgfkeys{tikz/mymatrixbrace/.style={decorate,thick}}

\newtheorem{theorem}{Theorem}[section]
\newtheorem{lemma}[theorem]{Lemma}
\newtheorem{corollary}[theorem]{Corollary}
\newtheorem{proposition}[theorem]{Proposition}
\usepackage{amssymb}
\usepackage{setspace}
\newtheorem{definition}[theorem]{Definition}
\newtheorem{example}[theorem]{Example}
\newtheorem{remark}[theorem]{Remark}

\title[Homotopy category over gentle algebras and Poset representations]{From the homotopy category of projective modules over gentle algebras to Poset representations}
\date{\today}

\author[Germán Benitez]{Germán Benitez}
\address{\noindent Departamento de Matem\'atica, Instituto de Ciências Exatas, Universidade Federal do Amazonas,  Manaus AM, Brazil}
\email{gabm03@gmail.com}

\author[Gustavo Costa]{Gustavo Costa}
\address{\noindent Centro de Matem\'atica, Computa\c{c}\~ao e Cogni\c{c}\~ao, Universidade Federal do ABC, Santo Andr\'e SP, Brazil}
\email{costagustt@gmail.com}

\begin{document}
\begin{abstract}
In \cite{BCP24}, the authors describe a triangulated structure of a quotient of a certain category of representations of posets, nowadays known as the Bondarenko's category. This category was essential in \cite{BM03} for classify all indecomposable objects of the derived category of gentle algebras. In view of this connection with the derived category, which possess a triangulated structure.  In this paper, we identify another triangulated structure for Bondarenko’s category, allowing us to utilize the functor  presented in \cite{BM03}.  This fucntor will  establishes a connection between the triangulated structure of the homotopy category of gentle algebras and the new triangulated structure of a quotient of a certain Bondarenko's category.
\end{abstract}

\allowdisplaybreaks
\maketitle
\tableofcontents
\section*{Introduction}
In 1975 V. Bondarenko in \cite{Bon75} define a certain   class of matrix with involution, at which was motivated by techniques called ``Self-Reproducibility'' established  	by L. A. Nazarova and A. V. Roiter in  \cite{NR73} for solve the Gelfand Problem (see \cite{Gel71}). This class of matrix are called nowadays of the  \emph{Bondarenko's matrix} are finite square block matrices $B=(B^{j}_i)_{i,j\in \mathcal{Y}}$ setting in $\Bbbk$, where $\mathcal{Y}$ is a linear ordered set with an involution $\sigma$ such that the following conditions hold:
\begin{enumerate}[(i)]
\item The number of rows
in each horizontal band $B_x$  is equal to the number of columns of each vertical band $B^x$ for all $x \in \mathcal{Y}$.
 
\item If $i,j \in \mathcal{Y}$ are such that $\sigma(i)=j$, then all matrices in $B_i$ (respectively, $B^i$) have the same number of rows (respectively, columns) as all matrices in $B_j$
 (respectively, $B^j$).
\item $B^2=0$.
\end{enumerate}

In view of this class of matrices, V. Bondarenko and Y. Drozd in \cite{BD82} define a relation between these classes of matrices. In other words, they describe morphisms between Bondarenko's matrices. In the sense, given two Bondarenko's matrices $B$ and  $C$, a \emph{morphism} from $B$ to $C$ is a block matrix $T=(T^{j}_{i})_{i,j \in \mathcal{Y}}$, with entries in $\Bbbk$ such that the following conditions hold:
\begin{enumerate}[(a)]
 \item The number of rows in each horizontal band $T_x$ (respectively, $C_x$) is equal to the number of columns of each vertical band $B^x$ (respectively, $T^x$) for all $x \in \mathcal{Y}$.
\item $TC=BT$.
\item If $i>j$, then $T^j_i=0$, where $<$ is the order relation in the poset $\mathcal{Y}$, i.e., all blocks below the main diagonal are 0.
\item If $\sigma(i)=j$, then $T^i_i=T^j_j$.
\end{enumerate}

 The \emph{Bondarenko's category} which the objects are Bondarenko's matrices and morphisms are as above, this category will be denoted by $s(\mathcal{Y},\Bbbk)$.

V. Bekkert and H. Merklen  in \cite{BM03}, found a connection between the derived category of a gentle algebra and a matrix problem presented by V. M. Bondarenko
\cite{Bon75}. Them show that the problem of finding the indecomposable objects of the derived category can be reduced to finding the indecomposable objects in the matrix problem. This techniques, was  adapted by V. Bekkert, E. N. Marcos and H. Merklen to classify the indecomposable objects of the derived category of Skewed-Gentle algebras in \cite{BNM03}, also by H. Giraldo and J. A. Vélez-Marulanda to describe Auslander-Reiten quiver of an algebra of dihedral type in \cite{GV16}. Moreover, the techniques used by V. Bekkert and H. Merklen were the starting point to research new classes of algebras in which is possible to describe indecomposable objects of its derived category, such algebras was introduced by  A. Franco, H. Giraldo and P. Rizzo \cite{FGR21} and called string almost gentle (SAG) algebras and SUMP algebras.

In light of its connection with the derived category, which possesses a triangulated structure, the authors G. Benitez, G. Costa, and L. Q. Pinto in  \cite{BCP24} introduce the quotient category $\kappa(\mathcal{Y},\Bbbk)=s(\mathcal{Y},\Bbbk)/\equiv$  and demonstrate that this category also has a triangulated structure. In this paper, we present a different triangulated structure for another quotient of Bondarenko's category, specifically for posets of the form $\mathcal{Y} \times \mathbb{Z}$.  Additionally, we utilize the functor defined in \cite{BM03} to show the existence of a triangulated functor from certain  homotopy  category of gentle algebras to a specific Bondarenko's category.

The  sections in this paper is structured as follows. In Section~\ref{sec:Bond} we introduce the Bondarenko's category associated to poset $\mathscr{Y}$ and we present the necessary notations used throughout the article, along with somebackground information relevant to the topic. Section \ref{sec:2} is introduced the standard $\mathcal{K}$-triangles which will be useful to define the family of distinguished triangles and contains the formulation of the main result with respectively proof. In Section \ref{sec:DG}, we define the homotopy category of gentle algebras and establish the existence of a triangulated functor from this category to a specific Bondarenko's category.




\section{ Bondarenko's category associated to poset \texorpdfstring{$\mathcal{Y} \times \mathbb{Z}$}{YxZ}}\label{sec:Bond}

In this paper we will denote by $\mathbb{Z}$ the set of integers, by $s(\mathcal{Y},\Bbbk)$ the Bondareko's category (as was defined in the introduction) for a  poset $\mathcal{Y}$ equipped with an involution $\sigma: \mathcal{Y}\longrightarrow\mathcal{Y}$. We are interesting to study certain quotient of $s(\mathcal{Y}\times \mathbb{Z},\Bbbk)$  (See Section \ref{sec:DG}), where the poset   
   $\mathcal{Y} \times \mathbb{Z}$ is equipped with the  anti-lexicographically order,  this means that
\begin{center}
    $[u,i]<[v,j]$ if and only if $i<j$ or $(i=j$ and $u<v)$,
\end{center}
 with involution $\sigma_{\mathcal{Y}\times \mathbb{Z}}$ on $\mathcal{Y}\times \mathbb{Z}$ is given by
 \begin{center}
 $\sigma_{\mathcal{Y}\times \mathbb{Z}}([u,i])=[v,j]$ if and only if $i=j$ and $\sigma(u)=v$.    
 \end{center}
 
Throughout this paper, for simplicity, the poset  $\mathcal{Y}\times \mathbb{Z}$ will be denoted by $\mathscr{Y}$ and by abuse of notation,  we will write $\sigma$ instead of $\sigma_{\mathcal{Y}\times \mathbb{Z}}$. The composition of two morphisms $f:X\longrightarrow Y$ and $g:Y\longrightarrow Z$ in a given category is denoted by $fg$.
 
Let us to start introducing the autofunctor $\llbracket  - \rrbracket:s(\mathscr{Y},\Bbbk)\longrightarrow s(\mathscr{Y},\Bbbk)$ given by $\llbracket B \rrbracket$ for objects and $\llbracket T \rrbracket$ for morphisms,  which are   defined by 
\begin{center}
$ \llbracket B \rrbracket_{[u,i]}^{[v,j]}=-B_{[u,i+1]}^{[v,j+1]}$ \ \ \    and \ \ \  $\llbracket T	 \rrbracket_{[u,i]}^{[v,j]}=T_{[u,i+1]}^{[v,j+1]}$, \ \ \    for all ${[u,i]},{[v,j]} \in \mathscr{Y}$.
\end{center}
Specifically,  the functor displace the horizontal band upward  and displace the vertical band for left.  

Set the matrix ${\Omega_T}$, for each morphism $T:B\longrightarrow C$ in $s(\mathscr{Y},\Bbbk)$,  
 where the  $({[u,i]},{[v,j]} )$th block has the following form
\begin{equation*}\label{eq4.1}
{\Omega_T}_{[u,i]}^{[v,j]}=\left(\begin{array}{cc}
    \llbracket B \rrbracket_{[u,i]}^{[v,j]} &T_{[u,i]}^{[v,j]}  \\ 
    (\mathbf{0}^{\llbracket B \rrbracket}_C)_{[u,i]}^{[v,j]} & C_{[u,i]}^{[v,j]}
    \end{array}\right)=\left(\begin{array}{cc}
    -B_{[u,i+1]}^{[v,j+1]} &T_{[u,i+1]}^{[v,j]}   \\
    (\mathbf{0}^{B}_C)_{[u,i]}^{[v,j+1]} & C_{[u,i]}^{[v,j]}
    \end{array}\right),
\end{equation*} 
with $\mathbf{0}^{\llbracket B \rrbracket}_{C}$ being the null morphism in $\textup{Hom}_{s(\mathscr{Y},\Bbbk)}(\llbracket B \rrbracket,C)$, sometimes can be  simplying $\mathbf{0}$ when no confusion. 

For a better exposition, the elements $[u,i]$ in $\mathscr{Y}$ will be denoted by $u_i$. And, for each object $B$ in $s(\mathscr{Y},\Bbbk)$, we define the \emph{support of} $B$ by the set $\mathscr{Y}_B := \{u_i \in \mathscr{Y}\mid  B_{u_i}\not=\emptyset\}$. In this way, the $(u_i, v_j )$th block has the form
	$$
    {\Omega_T}_{u_i}^{v_j}=\left(\begin{array}{cc}
    -B_{u_{i+1}}^{v_{j+1}} &T_{u_{i+1}}^{v_j}  \\ 
    (\mathbf{0}^{B}_C)_{u_i}^{v_{j+1}} & C_{u_i}^{v_j}
    \end{array}\right)
    $$
and, for each ${u_{i}},{v_{j}} \in \mathscr{Y}_{{\Omega_T}}=\mathscr{Y}_{\llbracket B \rrbracket}\cup \mathscr{Y}_{C}$ we can visualize such blocks in the following table

\vspace{.3cm}

\begin{center}
\begin{tabular}{|c||c|c|c|}
\hline
\mbox{$\left(\begin{array}{cc}
    -B_{u_{i+1}}^{v_{j+1}} &T_{u_{i+1}}^{v_{j}} \\
    (\mathbf{0}^B_C)_{u_{i}}^{v_{j+1}}  & C_{u_{i}}^{v_{j}} 
    \end{array}\right)$}  & \mbox{$v_j\in \mathscr{Y}_{\llbracket B \rrbracket}\cap \mathscr{Y}_{C}$} & \mbox{$v_j\in \mathscr{Y}_{\llbracket B \rrbracket}\setminus \mathscr{Y}_{C}$} &\mbox{$v_j\in \mathscr{Y}_{C}\setminus \mathscr{Y}_{\llbracket B \rrbracket}$} \\
\hline
\hline
$u_i\in \mathscr{Y}_{\llbracket B \rrbracket}\cap \mathscr{Y}_{C}$ & $\left(\begin{array}{cc}
    -B_{u_{i+1}}^{v_{j+1}}  &T_{u_{i+1}}^{v_{j}}   \\
    (\mathbf{0}^B_C)_{u_{i}}^{v_{j+1}}  & C_{u_{i}}^{v_{j}} 
    \end{array}\right)$	& $\left(\begin{array}{cc}
-B_{u_{i+1}}^{v_{j+1}}  \\
(\mathbf{0}^B_C)_{u_{i}}^{v_{j+1}} 	
\end{array}\right)$ 	&		$\left(\begin{array}{cc}
T_{u_{i+1}}^{v_{j}} \\ 
 C_{u_{i}}^{v_{j}} 
\end{array}\right)$	  \\
&&&\\
$u_i\in \mathscr{Y}_{\llbracket B \rrbracket}\setminus \mathscr{Y}_{C}$     				&  $\left(\begin{array}{cc}
-B_{u_{i+1}}^{v_{j+1}}  & T_{u_{i+1}}^{v_{j}} \\	
\end{array}\right)$	 & $\left(\begin{array}{cc}
-B_{u_{i+1}}^{v_{j+1}}  
\end{array}\right)$	 & 		$\left(T_{u_{i+1}}^{v_{j}} \right)$	 \\ 
&&&\\
$u_i\in \mathscr{Y}_{C}\setminus \mathscr{Y}_{\llbracket B \rrbracket}$ 			&  $\left(\begin{array}{cc}
(\mathbf{0}^B_C)_{u_{i}}^{v_{j+1}} & C_{u_{i}}^{v_{j}} 
\end{array}\right)$ &$\left((\mathbf{0}^B_C)_{u_{i}}^{v_{j+1}} \right)$&$\left(\begin{array}{cc}
C_{u_{i}}^{v_{j}} 	
\end{array}\right)$ \\

\hline
\end{tabular}
\end{center}


\vspace{.3cm}

\begin{lemma}
\label{lem:CT}
For any morphism $T\in\textup{Hom}_{s(\mathscr{Y},\Bbbk)}(B,C)$
	\begin{enumerate}[(i)]
	\item ${\Omega_T}$ is an object in $s(\mathscr{Y},\Bbbk)$.
	\item There exist morphisms 
$\iota\in \textup{Hom}_{s(\mathscr{Y},\Bbbk)}(C,{\Omega_T})$ and $\pi\in \textup{Hom}_{s(\mathscr{Y} ,\Bbbk)}({\Omega_T},\llbracket B \rrbracket) 
$ given by 
	$$
	\iota=\left(\begin{array}{cc}
    (\mathbf{0}_C^B)_{u_{i}}^{v_{j+1}} & (\textup{Id}_{C})_{u_{i}}^{v_{j}} \\
\end{array}\right)_{{u_{i}},{v_{j}}\in \mathscr{Y}}
	\ \ \ \ \text{ and } \ \ \ \ \ 
    \pi=\left(\begin{array}{cc}
    (\textup{Id}_{B})_{u_{i+1}}^{v_{j+1}}\\
    (\mathbf{0}_C^B)_{u_{i}}^{v_{j+1}}
\end{array}\right)_{{u_{i}},{v_{j}} \in \mathscr{Y}}.
	$$
	\end{enumerate}
\end{lemma}


%
%

\begin{proof} 
To demonstrate that ${\Omega_T}$ is indeed an object, the properties involving the compatibility of partitions and involution are inherent by construction. Furthermore, $\Omega_T^2=0$ follows directly from the calculus in each $(u_i,v_j)$th block
  {\footnotesize{ \begin{align*}
   (\Omega_T^2)_{u_{i}}^{v_{j}}& =\sum_{w_k\in \mathcal{Y}}\left(\begin{array}{cc}
   -B_{u_{i+1}}^{w_{k}}& T_{u_{i+1}}^{w_{k}}  \\
    (\mathbf{0}_C^B)_{u_{i}}^{w_{k}} & C_{u_{i}}^{w_{k}}
\end{array}\right) \left(\begin{array}{cc}
   -B_{w_{k}}^{v_{j+1}}&T_{w_{k}}^{v_{j}}\\
    (\mathbf{0}_C^B)_{w_{k}}^{v_{j+1}}& C_{w_{k}}^{v_{j}}
\end{array}\right)
\\&= \left(\begin{array}{cc}
  \sum\limits_{w_k\in \mathscr{Y}} B_{u_{i}}^{w_{k}}B_{w_{k}}^{v_{j+1}}&\sum\limits_{w_k\in \mathscr{Y}}(T_{u_{i}}^{w_{k}}C_{w_{k}}^{v_{j+1}}-B_{u_{i}}^{w_{k}}T_{w_{k}}^{v_{j+1}})  \\
   (\mathbf{0}^B_C)_{u_{i}}^{v_{j+1}}&\sum\limits_{w_k\in \mathscr{Y}} C_{u_{i}}^{w_{k}}C_{w_{k}}^{v_{j+1}}
\end{array} \right)\\
 &= \left(\begin{array}{cc}
  (B^2)_{u_{i+1}}^{v_{j+1}}  & (TC-BT)_{u_{i+1}}^{v_{j}}\\
   (\mathbf{0}^B_C)_{u_{i}}^{v_{j+1}} & (C^2)_{u_{i}}^{v_{j}}
\end{array} \right) 
   \end{align*}}}
as $B^2=0$, $C^2=0$, $TC=BT$, then  $(\Omega_T^2)_{u_{i}}^{v_{j}}=0$, for all ${u_{i}},{v_{j}}\in \mathscr{Y}$.

	It suffices to demonstrate  for $\iota$, the proof for $\pi$ is analogous. Properties (a), (c), and (d) from the definition of a morphism in $s(\mathscr{Y},\Bbbk)$ follow directly { from the latter construction}. The equality $C\iota = \iota\Omega_T$ results from the following computations  in each $(u_i,v_j)$th block
{\footnotesize{\begin{align*}
  (  C\iota)^{v_j}_{u_i} &=\sum_{w_kk\in \mathscr{Y}}C^{w_k}_{u_i}\left(\begin{array}{cc}
    (\mathbf{0}_C^B)_{w_{k}}^{v_{j+1}} & (\textup{Id}_{C})_{w_{k}}^{v_{j}} \\
\end{array}\right)=C^{v_j}_{u_i}\left(\begin{array}{cc}
    (\mathbf{0}_C^B)^{v_{j+1}}_{v_j} & (\textup{Id}_{C})^{v_j}_{v_j} \\
\end{array}\right)=\left(\begin{array}{cc}
    (\mathbf{0}_C^B)_{u_i}^{v_{j+1}} & C^{v_j}_{u_i}   \\
\end{array} \right)\\
&= \left(\begin{array}{cc}
    (\mathbf{0}_C^B)_{u_i}^{u_{i+1}} & (\textup{Id}_{C})^{u_i}_{u_i}  \\
\end{array}\right) \left(\begin{array}{cc}
    -B_{u_{i+1}}^{v_{j+1}} &T_{u_{i+1}}^{v_j}  \\ 
    (\mathbf{0}^{B}_C)_{u_i}^{v_{j+1}} & C_{u_i}^{v_j}
    \end{array}\right)=
\sum_{w_k \in \mathscr{Y}}\left(\begin{array}{cc}
    (\mathbf{0}_C^B)_{u_i}^{w_k} & (\textup{Id}_{C})^{w_k}_{u_i}   \\
\end{array}\right)\left(\begin{array}{cc}
    -B_{w_{k}}^{v_{j+1}} &T_{w_{k}}^{v_j}  \\ 
    (\mathbf{0}^{B}_C)_{w_{k}}^{v_{j+1}} & C_{w_{k}}^{v_j}
    \end{array}\right)\\
&=(\iota \Omega_T)^{v_j}_{u_i}.
\end{align*}}}

\end{proof}

From now on, in this paper we will denote the morphism $\iota$ and $\pi$ of the latter lemma by $\iota_C$ and $\pi_{\llbracket B \rrbracket}$ respectively. It is worth pointing out that the $({u_{i}},{v_{j}})$th block
    $$
    ({\iota}_C)_{u_{i}}^{v_{j}}=\left(\begin{array}{cc}
     (\mathbf{0}^B_C)_{u_{i}}^{v_{j+1}} &(\textup{Id}_{C})_{u_{i}}^{v_{j}}  \\
\end{array}\right)
    $$
has $r(C_{u_i})$ rows and $c(\llbracket B \rrbracket^{v_j})+c(C^{v_j})$ columns, where $r(D_{u_i})$ (respectively, $c(D^{u_i})$) denote the number of rows (respectively, columns) in the horizontal (respectively, vertical) band $D_{u_i}$ (respectively,  $D^{u_i}$). Specifically, for $u_i\in \mathscr{Y}_{C}$ and $v_j\in \mathscr{Y}_{\llbracket B \rrbracket}\cup \mathscr{Y}_{C}$

\vspace{.3cm}

\begin{center}
\begin{tabular}{|c||c|c|c|}
\hline
\mbox{$\left(\begin{array}{cc}
(\mathbf{0}^B_C)_{u_{i}}^{v_{j+1}} & (\textup{Id}_C)_{u_{i}}^{v_{j}} 
\end{array}\right)$}  & \mbox{$v_j\in \mathscr{Y}_{\llbracket B \rrbracket}\cap \mathscr{Y}_{C}$} & \mbox{$v_j\in \mathscr{Y}_{\llbracket B \rrbracket}\setminus \mathscr{Y}_{C}$} & \mbox{$v_j\in \mathscr{Y}_{C}\setminus \mathscr{Y}_{\llbracket B \rrbracket}$}  \\
\hline
\hline

$u_i\in \mathscr{Y}_{C}$ 			&  $\left(\begin{array}{cc}
(\mathbf{0}_C^B)_{u_{i}}^{v_{j+1}} & (\textup{Id}_C)_{u_{i}}^{v_{j}} 
\end{array}\right)$ & $(\mathbf{0}_C^B)_{u_i}^{v_{j+1}}$& $(\textup{Id}_C)_{u_{i}}^{v_{j}}$ 
\\
\hline
\end{tabular}
\end{center}
Similarly, for the $(u_i,v_j)$th block  $\pi _{\llbracket B \rrbracket}$. 

\vspace{,2cm}

For a better exposition, in the examples \ref{ex:triangulo}, \ref{example 3.4} and \ref{exem3.6} we will use color blue to identify the indices of a matrix block.


\begin{example}
\label{ex:triangulo}
  Let us consider the poset $\mathcal{Y}=\left\{u<a<v<b\right\}$ with involution $\sigma$ given by
    $$
    \sigma(u)=v    
    \ \ \ \ \text{and }\ \ \ \ 
    \sigma(a)=b. 
    $$
For the following morphism $T:B\longrightarrow \llbracket B\rrbracket$ in $s(\mathscr{Y},\Bbbk)$
$$
{\footnotesize 
	\xymatrix{
	B={\left(\begin{array}{c||c|c|c|c}
     &{\blue u_1}&{\blue v_1} & {\blue a_2}&{\blue b_2}\\
     \hline\hline
    {\blue u_1} &0&0&-1&0\\ \hline
    {\blue v_1} &0&0&0&1\\
    \hline
    {\blue a_2} &0&0&0&0\\
    \hline
    {\blue b _2} &0&0&0&0 
    \end{array}\right)} \ar[rrrrr]^{T={\left(\begin{array}{c||c|c|c|c}
     &{\blue u_0}&{\blue v_0} & {\blue a_1}&{\blue b_1}\\
     \hline\hline
    {\blue u_1} &0&0&1&0\\ \hline
    {\blue v_1} &0&0&0&1\\
    \hline
    {\blue a_2} &0&0&0&0\\
    \hline
    {\blue b _2} &0&0&0&0 
    \end{array}\right)}}&&&&&
  \llbracket B\rrbracket={\left(\begin{array}{c||c|c|c|c}
     &{\blue u_0}&{\blue v_0} & {\blue a_1}&{\blue b_1}\\
     \hline\hline
    {\blue u_0} &0&0&1&0\\ \hline
    {\blue v_0} &0&0&0&-1\\
    \hline
    {\blue a_1} &0&0&0&0\\
    \hline
    {\blue b_1} &0&0&0&0 
    \end{array}\right)}}}
    $$
the object $\Omega_{T}$ is described of the following matrix
$$
 {\footnotesize{\Omega_{T}=\left(\begin{array}{c||cc|cc|cc|cc}
     &{\blue u_0}&&{\blue v_0} && {\blue a_1}&&{\blue b_1}\\
     \hline\hline
    {\blue u_0}&0&0&0&0&1&1&0&0\\
   			 &0&0&0&0&0&1&0&0\\ 
    \hline
   		 {\blue v_0}&0&0&0&0&0&0&-1&1\\
   			 &0&0&0&0&0&0&0&-1\\ 
    \hline
    {\blue a_1} &0&0&0&0&0&0&0&0\\
   			 &0&0&0&0&0&0&0&0\\ 
   			 \hline
    {\blue b_1}  &0&0&0&0&0&0&0&0\\
   		&0&0&0&0&0&0&0&0\\
 \end{array}\right)}}
$$
In this way we obtain of the following sequence of morphisms:
$$
	{\footnotesize
	\xymatrix{B\ar[r]_-{T}&{\llbracket B \rrbracket}\ar[rrrrrr]_-{\iota_{\llbracket B \rrbracket}={{\left(\begin{array}{c||c|c|c|c}
     &{\blue u_0}&{\blue v_0} & {\blue a_1}&{\blue b_1}\\
     \hline\hline
    {\blue u_0}&0\ \ 1&0 \ \ 0&0 \ \ 0&0\ \ 0\\
    \hline
   		 {\blue v_0}&0\ \ 0&0 \ \ 1&0 \ \ 0&0\ \ 0\\
    \hline
    {\blue a_1} &0\ \ 0&0 \ \ 0&0 \ \ 1&0\ \ 0\\ 
   			 \hline
    {\blue b_1}  &0\ \ 0&0 \ \ 0&0 \ \ 0&0\ \ 1
 \end{array}\right)}}}&&&&&&\Omega_T\ar[rrrr]_-{\pi_{\llbracket B\rrbracket}={{\left(\begin{array}{c||c|c|c|c}
     &{\blue u_0}&{\blue v_0} & {\blue a_1}&{\blue b_1}\\
     \hline\hline
    {\blue u_0}&1&0&0&0\\ 
&0&0&0&0\\
    \hline
   		 {\blue v_0}&0&1&0&0\\
    &0&0&0&0\\
    \hline
    {\blue a_1} &0&0&1&0\\
   			 &0&0&0&0\\ 
   			 \hline
    {\blue b_1}  &0&0&0&1\\
   		&0&0&0&0  
 \end{array}\right)}}}&&&& \llbracket B\rrbracket .}}$$ 
  
\end{example}


We will conclude this section with two technical lemmas that serve as further examples of morphisms in $s(\mathscr{Y},\Bbbk)$. These lemmas will be applied in propositions \ref{prop:TR2} and \ref{prop:Octae}, respectively.
\begin{lemma}\label{tr2:lemma}
   For any morphism $T\in\textup{Hom}_{s(\mathscr{Y},\Bbbk)}(B,C)$, there exist morphisms $R\in\textup{Hom}_{s(\mathscr{Y},\Bbbk)}(\llbracket B \rrbracket,{\Omega_{\iota_C}}) $ and $S\in\textup{Hom}_{s(\mathscr{Y},\Bbbk)}({\Omega_{\iota_C}},\llbracket B \rrbracket)$ given by
    $$
    R=\left(
\begin{array}{ccc}
    -T_{u_{i+1}}^{v_{j+1}} & (\textup{Id}_{B})_{u_{i+1}}^{v_{j+1}}& (\mathbf{0}^{C}_{B})_{u_{i+1}}^{v_{j}} \\
\end{array}
 \right)_{{u_{i}},{v_{j}}\in \mathscr{Y}}\ \  \ \ \ \ \text{ and } \ \ \ \ \ \   	S=\left(
\begin{array}{cc}
    (\mathbf{0}_C^{B})_{u_{i+1}}^{v_{j+1}}\\ 
    (\textup{Id}_{B})_{u_{i+1}}^{v_{j+1}}\\
    (\mathbf{0}_C^{B})_{u_{i}}^{v_{j+1}}
\end{array}
 \right)_{{u_{i}},{v_{j}}\in \mathscr{Y}}.$$
\end{lemma}
\begin{proof}
We will only show that $R\in\textup{Hom}_{s(\mathscr{Y},\Bbbk)}(\llbracket B\rrbracket,\Omega_{\iota_C})$, because for $S$ the proof is analogous.
 
The properties (a), (c) and (d) of definition of morphism in $s(\mathscr{Y},\Bbbk)$,  follows from the fact that $\llbracket T\rrbracket$ and $\textup{Id}_{\llbracket B\rrbracket}$ are morphisms in $s(\mathscr{Y},\Bbbk)$. The equality $\llbracket B\rrbracket R=R\Omega_{\iota_C}$ is consequence of the following computations 
{\footnotesize{\begin{align*}
      ( \llbracket B\rrbracket R)_{u_i}^{v_j}
      &=\sum_{w_k\in \mathscr{Y}} -B_{u_{i+1}}^{w_k}\left(
\begin{array}{ccc}
    -T_{w_k}^{v_{j+1}} & (\textup{Id}_{B})_{w_k}^{v_{j+1}}& (\mathbf{0}^{C}_{B})_{w_k}^{v_j} \\
\end{array}
 \right)= \left(
\begin{array}{ccc}
    (BT)_{u_{i+1}}^{v_{j+1}} & -B_{u_{i+1}}^{v_{j+1}} & (\mathbf{0}^{C}_{B})_{u_{i+1}}^{v_j} \\
\end{array}
 \right)\\
 &=\sum_{w_k\in \mathscr{Y}} \left(
\begin{array}{ccc}
    -T_{u_{i+1}}^{w_k} & (\textup{Id}_{B})_{u_{i+1}}^{w_k}& (\mathbf{0}^{C}_{B})_{u_{i+1}}^{w_k} \\
\end{array}
 \right)\left(\begin{array}{cccc}
  -C_{w_k}^{v_{j+1}}   & (\mathbf{0}_{C}^{B})_{w_k}^{v_{j+1}}&  (\textup{Id}_C)_{w_k}^{v_j} \\ 
(\mathbf{0}_{B}^{C})_{w_k}^{v_{j+1}} & -B_{w_k}^{v_{j+1}} & T_{w_k}^{v_j} \\
    (\mathbf{0}_{C}^{C})_{w_k}^{v_{j+1}}  & (\mathbf{0}_{C}^{B})_{w_k}^{v_{j+1}} & C_{w_k}^{v_j}    
\end{array}\right)=(R\Omega_{\iota_C})_{u_i}^{v_j}
\end{align*}}}
for all $u_i,v_j \in \mathscr{Y}$.

\end{proof}
\begin{lemma}\label{tr6:lemma}
For any morphisms $S \in \textup{Hom}_{s(\mathscr{Y},\Bbbk)}(B,C)$ and $T \in \textup{Hom}_{s(\mathscr{Y},\Bbbk)}(C,D)$, there exist morphisms $F\in \textup{Hom}_{s(\mathscr{Y},\Bbbk)}(\Omega_S,\Omega_{ST})$, $G\in \textup{Hom}_{s(\mathscr{Y},\Bbbk)}(\Omega_{ST},\Omega_T)$ and $\Lambda\in \textup{Hom}_{s(\mathscr{Y},\Bbbk)}(\Omega_{T},\Omega_F)$ given by
 {\small{  $$
F = \left(\begin{array}{cc}
   (\textup{Id}_{B})_{u_{i+1}}^{v_{j+1}} & (\mathbf{0}_{B}^D)_{u_{i+1}}^{v_j}  \\ 
    (\mathbf{0}_C^{B})_{u_i}^{v_{j+1}} & T_{u_i}^{v_j}  
\end{array}\right)_{{u_i},{v_j} \in \mathscr{Y}}
\ \ \ \ \ \text{, }\ \ \ \ \ G = \left(\begin{array}{cc}
   S_{u_{i+1}}^{v_{j+1}}  & (\mathbf{0}_B^D)_{u_{i+1}}^{v_j}  \\ 
    (\mathbf{0}_D^{C})_{u_i}^{v_{j+1}} & (\textup{Id}_D)_{u_i}^{v_j}  
\end{array}\right)_{{u_i},{v_j} \in \mathscr{Y}}$$}}
and
$$
  \Lambda=\left(\begin{array}{ccccc}
     (\mathbf{0}_{C}^{B})_{u_{i+1}}^{v_{j+2}}& (\textup{Id}_{C})_{u_{i+1}}^{v_{j+1}} &
     (\mathbf{0}_{C}^{B})_{u_{i+1}}^{v_{j+1}}& (\mathbf{0}_{C}^{D})_{u_{i+1}}^{v_j}\\ 
     (\mathbf{0}_{D}^{B})_{u_i}^{v_{j+2}}& (\mathbf{0}_D^{C})_{u_i}^{v_{j+1}} &
     (\mathbf{0}_{D}^{B})_{u_i}^{v_{j+1}} & (\textup{Id}_D)_{u_i}^{v_j} 
\end{array}\right)_{{u_i},{v_j}\in \mathscr{Y}}. 
$$
\end{lemma}
\begin{proof}
We show that $F\in\textup{Hom}_{s(\mathscr{Y},\Bbbk)}({\Omega_S}, {\Omega_{ST}})$, the proof for $G\in\textup{Hom}_{s(\mathscr{Y},\Bbbk)}({\Omega_{ST}}, {\Omega_{T}})$ is analogous. The properties (a), (c) and (d) of the definition of morphism in $s(\mathscr{Y},\Bbbk)$,  follows from the fact that $T$ and $\textup{Id}_{\llbracket B \rrbracket}$ are morphisms in $s(\mathscr{Y},\Bbbk)$. The equality ${\Omega_S}F=F{\Omega_{ST}}$ is  consequence of the following computations 
{\small{\begin{align*}
   ( {\Omega_{S}}F)_{u_i}^{v_j}&= \sum_{w_k\in \mathscr{Y}}\left(\begin{array}{cc}
   -B_{u_{i+1}}^{w_k}  & S_{u_{i+1}}^{w_k}  \\ 
    (\mathbf{0}^{B}_C)_{u_i}^{w_k} & C_{u_i}^{w_k}  
\end{array}\right)\left(\begin{array}{cc}
   (\textup{Id}_{B})_{w_k}^{v_{j+1}} & (\mathbf{0}^D_B)_{w_k}^{v_j}  \\ 
    (\mathbf{0}_C^{B})_{w_k}^{v_{j+1}} & T_{w_k}^{v_j}  
\end{array}\right)
=\left(\begin{array}{cc}
   -B_{u_{i+1}}^{v_{j+1}} & (ST)_{u_{i+1}}^{v_j}  \\ 
    (\mathbf{0}_C^{B})_{u_i}^{v_{j+1}}& (CT)_{u_i}^{v_j}  
\end{array}\right),\\
   ( F{\Omega_{ST}})_{u_i}^{v_j}&= \sum_{w_k \in \mathscr{Y}}\left(\begin{array}{cc}
   (\textup{Id}_{B})_{u_{i+1}}^{w_k} & (\mathbf{0}^D_B)_{u_{i+1}}^{w_k}  \\ 
    (\mathbf{0}_C^{B})_{u_i}^{w_k} & T_{u_i}^{w_k} 
\end{array}\right)\left(\begin{array}{cc}
   -B_{w_k}^{v_{j+1}}   & (ST)_{w_k}^{v_j}   \\ 
    (\mathbf{0}^{B}_D)_{w_k}^{v_{j+1}} & D_{w_k}^{v_j}   
\end{array}\right)=\left(\begin{array}{cc}
   -B_{u_{i+1}}^{v_{j+1}}  & (ST)_{u_{i+1}}^{v_j} \\ 
     (\mathbf{0}^{B}_C)_{u_i}^{v_{j+1}} & (TD)_{u_i}^{v_j}  
\end{array}\right)
\end{align*}}}
for all $u_i,v_j \in \mathscr{Y}$.

 Finally, let us show that $
    \Lambda\in\textup{Hom}_{s(\mathscr{Y},\Bbbk)}({\Omega_T},{\Omega_F}).
    $ The properties (a), (c) and (d) of the definition of morphism in $s(\mathscr{Y},\Bbbk)$,  follows from the fact that $\textup{Id}_{C}$ and $\textup{Id}_{D}$  are morphisms in $s(\mathscr{Y},\Bbbk)$. The equality $\Omega_T\Lambda=\Lambda\Omega_{F}$ is a consequence of the following computations 
{\small{\begin{align*}
    (&\Lambda{\Omega_{F}} )_{u_i}^{v_j}\\ &=\sum_{w_k\in \mathscr{Y}}\!\!\left(\begin{array}{ccccc}
     (\mathbf{0}_{C}^{B})_{u_{i+1}}^{w_k} & (\textup{Id}_{C})_{u_{i+1}}^{w_k} &
     (\mathbf{0}_{C}^{B})_{u_{i+1}}^{w_k} &(\mathbf{0}_{C}^{D})_{u_{i+1}}^{w_k} \\
     (\mathbf{0}_{D}^{B})_{u_i}^{w_k} & (\mathbf{0}_D^{C})_{u_i}^{w_k} &
     (\mathbf{0}_{D}^{B})_{u_i}^{w_k} & (\textup{Id}_D)_{u_i}^{w_k} 
\end{array}\right)\!\!\left(\begin{array}{cc}
      ( \llbracket{\Omega_S\rrbracket})_{w_k}^{v_j}    & F_{w_k}^{v_j} \\
        (\mathbf{0}^{\llbracket{\Omega_S}\rrbracket}_{{\Omega_{ST}}})_{w_k}^{v_j}  & ({\Omega_{ST}})_{w_k}^{v_j}
     \end{array}\right)\\
    &=\sum_{w_k\in \mathscr{Y}}\!\!\left(\begin{array}{ccccc}
     (\mathbf{0}_{C}^{B})_{u_{i+1}}^{w_k} & (\textup{Id}_{C})_{u_{i+1}}^{w_k} &
     (\mathbf{0}_{C}^{B})_{u_{i+1}}^{w_k} &(\mathbf{0}_{C}^{D})_{u_{i+1}}^{w_k} \\ 
     (\mathbf{0}_{D}^{B})_{u_i}^{w_k} & (\mathbf{0}_D^{C})_{u_i}^{w_k} &
     (\mathbf{0}_{D}^{B})_{u_i}^{w_k} & (\textup{Id}_D)_{u_i}^{w_k} 
\end{array}\right)\!\!\left(\begin{array}{cccc}
        B_{w_k}^{v_{j+2}}  & -S_{w_k}^{v_{j+1}}& (\textup{Id}_{B})_{w_k}^{v_{j+1}} & (\mathbf{0}_B^D)_{w_k}^{v_j}  \\ 
          (\mathbf{0}_{C}^{B})_{w_k}^{v_{j+2}}  & -C_{w_k}^{v_{j+1}} & (\mathbf{0}_{C}^{B})_{w_k}^{v_{j+1}} & T_{w_k}^{v_j} \\ 
             (\mathbf{0}_{B}^{B})_{w_k}^{v_{j+2}} &(\mathbf{0}_{B}^{C})_{w_k}^{v_{j+1}} & -B_{w_k}^{v_{j+1}}& (ST)_{w_k}^{v_j}  \\ 
                 (\mathbf{0}_D^{B})_{w_k}^{v_{j+2}}& (\mathbf{0}_D^{C})_{w_k}^{v_{j+1}}&(\mathbf{0}_D^{B})_{w_k}^{v_{j+1}} &D_{w_k}^{v_j}  \\
     \end{array}\right)\\
     &=
\left(\begin{array}{ccccc}
     (\mathbf{0}_{C}^{B})_{u_{i+1}}^{v_{j+2}}& -C_{u_{i+1}}^{v_{j+1}} &
     (\mathbf{0}_{C}^{B})_{u_{i+1}}^{v_{j+1}}& T_{u_{i+1}}^{v_j}\\ 
     (\mathbf{0}_{D}^{B})_{u_i}^{v_{j+2}}& (\mathbf{0}_D^{C})_{u_i}^{v_{j+1}} &
     (\mathbf{0}_{D}^{B})_{u_i}^{v_{j+1}} & D_{u_i}^{v_j} 
\end{array}\right)   \\
&=\sum_{w_k\in \mathscr{Y}}\!\!\left(\begin{array}{cc}
    -C_{u_{i+1}}^{w_k} & T_{u_{i+1}}^{w_k}  \\ 
    (\mathbf{0}^{C}_D)_{u_i}^{w_k} &  D_{u_i}^{w_k}
\end{array}\right)\!\!\left(\begin{array}{ccccc}
     (\mathbf{0}_{C}^{B})_{w_k}^{v_{j+2}}& (\textup{Id}_{C})_{w_k}^{v_{j+1}} &
     (\mathbf{0}_{C}^{B})_{w_k}^{v_{j+1}}& (\mathbf{0}_{C}^{D})_{w_k}^{v_j}\\ 
     (\mathbf{0}_{D}^{B})_{w_k}^{v_{j+2}}& (\mathbf{0}_D^{C})_{w_k}^{v_{j+1}} &
     (\mathbf{0}_{D}^{B})_{w_k}^{v_{j+1}} & (\textup{Id}_D)_{w_k}^{v_j} 
\end{array}\right)\\
&=(\Omega_T\Lambda)_{u_i}^{v_j}
\end{align*}}}
for all $u_i,v_j \in \mathscr{Y}$.

\end{proof}

\section{ Triangulated structure for  \texorpdfstring{$\mathcal{K}(\mathcal{Y} \times \mathbb{Z},\Bbbk)$}{K(YxZ)}}\label{sec:2}

In this section, we introduce an appropriate triangulated structure for a certain quotient of the Bondarenko's category, which is slightly different to the quotient given in \cite{BCP24}, in the Example \ref{exem3.6} we show why another equivalence relation is needed. This triangulated structure will be used in Section \ref{sec:DG} to ensure the existence of a triangulated functor from a certain homotopy category to Bondarenko's category.  Let us start defining an equivalence relation $\simeq$ on
each Hom-set $\textup{Hom}_{s(\mathscr{Y},\Bbbk)}(B,C)$ by $S \simeq T$ if and only if there exist a matrix  $L$ (we will call \emph{$\mathcal{K}$-matrix}) satisfying: 

\begin{enumerate}
     \item[(i)] The horizontal (respectively, vertical) partition of $L$ is compatible with the vertical (respectively, horizontal) partition of $B$ (respectively, $C$).
     \item[(ii)] $S-T=BL+LC$.
     \item[(iii)] If $i>j+1$ or $(u>v$ and $i=j+1)$, then $L_{u_{i}}^{v_{j}}=0$. 
     \item[(iv)] If $\sigma(u_i)=v_j$, then $L_{u_{i}}^{u_{i}}=L_{v_{j}}^{v_{j}}$.

\end{enumerate}


\begin{remark}\label{remark2.1}
Let us mention three important points: 
\begin{enumerate}[(i)]
\item The condition $S-T=BL+LC$ implies that, in general, $L$ is not necessarily a morphism in $s(\mathscr{Y},\Bbbk)$.

\item The equivalence relation $\simeq$ is slightly different from the equivalence relation $\equiv$ defined in \cite[Section 3]{BCP24}, where $S \equiv T$ if and only if there exists a matrix $L$ (called a \emph{$\kappa$-matrix}) satisfying the same conditions \textup{(i)}, \textup{(ii)}, \textup{(iv)} and the condition \textup{(iii)} changed by 
\begin{center}
\textup{(iii')} If $i>j$ or $(u>v$ and $i=j)$, then $L_{u_{i}}^{v_{j}}=0$.
\end{center}
\item Note that $S\equiv T$ implies that $S\simeq T$, but the converse does not hold (see Example \ref{exem3.6}).
\end{enumerate} 
\end{remark}


Moreover, the morphisms equivalent to zero form a two-sided ideal in $s(\mathscr{Y},\Bbbk)$ (see \cite[Definition 3.1, p. 420]{ASS06} for the definition of ideal). To see this, let us consider the sets $\Sigma:=\left\{T \text{ morphism in } s(\mathscr{Y},\Bbbk)\mid T\simeq \mathbf{0}\right\}$ and $\Sigma(B, C) :=\Sigma \cap \textup{Hom}_{s(\mathscr{Y}, \Bbbk)}(B,C)$ where $B,C\in s(\mathscr{Y},\Bbbk)$, which are $\Bbbk$-vector spaces.


\begin{lemma}
\label{ideal}
$\Sigma$ is a two-sided ideal in $s(\mathscr{Y},\Bbbk).$ 
\end{lemma}


\begin{proof}
We show that $FG\in \Sigma(B, D)$, where $F \in \Sigma(B,C)$ and $G \in \textup{Hom}_{s(\mathscr{Y},\Bbbk)}(C, D)$. In fact, since $F \simeq \mathbf{0}$, there exists a $\mathcal{K}$-matrix $L$ such that $F=BL+LC$, hence
	$$
	FG = BLG + LCG = BLG + LGD = B\Tilde{L} + \Tilde{L}D,
	$$ 
where $\Tilde{L} := LG$ and the second equality follows of the fact that $CG = GD$. 

To see that $\Tilde{L}$ is a $\mathcal{K}$-matrix, note that the conditions (i), (ii), and (iv) of the definition of $\mathcal{K}$-matrix are straightforward, because $L$ is a $\mathcal{K}$-matrix and $G$ is a morphism. And, to verify that $\Tilde{L}^{v_j}_{u_i} = 0$ whenever $i > j + 1$ or $(u > v \text{ and } i = j + 1)$, we consider the expression 
	$$
	\tilde{L}^{v_j}_{u_i}=\sum_{w_k\in \mathcal{Y}\times\mathbb{Z}}L^{w_k}_{u_i}G^{v_j}_{w_k}=\sum_{\substack{i=k+1\\ w \in \mathcal{Y}}}L^{w_k}_{u_i}G^{v_j}_{w_k}+\sum_{\substack{i< k+1\\ w \in \mathcal{Y}}}L^{w_k}_{u_i}G^{v_j}_{w_k}+\sum_{\substack{i>k+1\\ w \in \mathcal{Y}}}L^{w_k}_{u_i}G^{v_j}_{w_k}.
	$$
In the following cases we will apply over $G$ the item (c) of the definition of morphism and over $L$ the item (iii) of the definition of $\mathcal{K}$-matrix. Therefore:

\begin{itemize}
\item If $i>j+1$, we have $G^{v_j}_{w_k}=0$ in the first two summand and $L^{w_k}_{u_i }=0$ in the third summand, which imply that $\tilde{L}^{v_j}_{u_i}=0$. 

\item On the other hand, assuming that $i=j+1$ and $u>v$, we have

\begin{align*}
\tilde{L}^{v_j}_{u_i} &= \sum_{\substack{ u\leq w \in \mathcal{Y}}}L^{w_j}_{u_{j+1}}G^{v_j}_{w_j}+\sum_{\substack{ u>w \in \mathcal{Y}}}L^{w_j}_{u_{j+1}}G^{v_j}_{w_j}+\sum_{\substack{i< k+1\\ w \in \mathcal{Y}}}L^{w_k}_{u_i}G^{v_j}_{w_k}+\sum_{\substack{i>k+1\\ w \in \mathcal{Y}}}L^{w_k}_{u_i}G^{v_j}_{w_k}
\end{align*}
which is zero, because in the first and third summand $G^{v_j}_{w_k} = 0$, while in the second and fourth summand $L^{w_k}_{u_i} = 0$. 

\end{itemize}		
Thus, $\tilde{L}$ is a $\mathcal{K}$-matrix and   consequently $FG \in \Sigma(B, D)$. Similarly, we can conclude that $HF\in \Sigma(D, C)$ for any $F \in \Sigma(B,C)$ and $H \in \textup{Hom}_{s(\mathscr{Y},\Bbbk)}(D, B)$. Therefore, $\Sigma$ is a two-sided ideal in $s(\mathscr{Y},\Bbbk)$.

\end{proof}


From Lemma \ref{ideal}, we can define the quotient category $\mathcal{K}(\mathscr{Y},\Bbbk)$ over $s(\mathscr{Y},\Bbbk)$, to be the category with the objects Bondarenko's matrices and the group of morphisms is given by $ \textup{Hom}_{\mathcal{K}(\mathscr{Y},\Bbbk)}(B,C) = {\textup{Hom}_{s(\mathscr{Y},\Bbbk)}(B ,C)}/\Sigma(B,C)$, for each $B,C\in\mathcal{K}(\mathscr{Y},\Bbbk)$.  It is easy to check that the autofunctor $\llbracket -\rrbracket$ preserves  the equivalence relation $\simeq$, therefore the induced autofunctor $\llbracket-\rrbracket:\mathcal{K}(\mathscr{Y},\Bbbk)\longrightarrow \mathcal{K}(\mathscr{Y},\Bbbk)$ is well-defined.
The main goal in this section, is to give a triangulated structure for $\mathcal{K}(\mathscr{Y},\Bbbk)$. To this end, let us define the $\mathcal{K}$-\emph{standard triangle} for any morphism $T \in \textup{Hom}_{s(\mathscr{Y},\Bbbk)}(B,C)$ by the sequence of morphisms
	$$
    \xymatrix{
    B\ar[r]^-{T} &
    C\ar[r]^-{\iota_C} &
    {\Omega_{T}}\ar[r]^-{\pi _{\llbracket B \rrbracket}} &
    \llbracket B \rrbracket},
    $$
where $\iota_C$ and $\pi_{\llbracket B \rrbracket }$ are the morphisms from Lemma \ref{lem:CT}. For an example of a $\mathcal{K}$-standard triangle, see Example~\ref{ex:triangulo}. Now, we can define the family of distinguished triangles $\mathcal{D}$ to be triangles of the form 
	$$
    \xymatrix{
    X\ar[r]^-{u} &
    Y\ar[r]^-{v} &
    Z\ar[r]^-{w} &
    \llbracket X \rrbracket}\ \ \ \ \text{in\ \ \ $\mathcal{K}(\mathscr{Y},\Bbbk)$}
    $$ 
which is isomorphic to a $\mathcal{K}$-standard triangle in $\mathcal{K}(\mathscr{Y},\Bbbk)$.  In others words, there exists an isomophism of triangles in $\mathcal{K}(\mathscr{Y},\Bbbk)$
    $$
    \xymatrix{
    X\ar[r]^-{u}\ar[d]_-{\cong} &
    Y\ar[r]^-{v}\ar[d]_-{\cong} &
    Z\ar[r]^-{w}\ar[d]_-{\cong} &
    \llbracket X\rrbracket\ar[d]_-{\cong}\\
    B\ar[r]_-{T} &
    C\ar[r]_-{\iota_C} &
    \Omega_T\ar[r]_-{\pi_{\llbracket B\rrbracket}} &
    \llbracket B\rrbracket}
    $$
for some morphism $T:B\longrightarrow C$ in $s(\mathscr{Y},\Bbbk)$.

For a better exposition, we will denote
the triangle $
    \xymatrix{
    X\ar[r]^-{u} &
    Y\ar[r]^-{v} &
    Z\ar[r]^-{w} &
    \llbracket X \rrbracket}
    $ 
by the sextuple $(X,Y,Z,u,v,w)$. 
See \cite[Section 12.3, pp. 303--309]{DW17} or \cite[Section 1.1, pp. 1--9]{Hap88} for the definition and properties of triangulated category. 
	
The main result in this section (Theorem \ref{main 1}) states that the category $\mathcal{K}(\mathscr{Y},\Bbbk)$ with the autofunctor $\llbracket - \rrbracket:\mathcal{K}(\mathscr{Y},\Bbbk)\longrightarrow \mathcal{K}(\mathscr{Y},\Bbbk)$ and the family of distinguished  triangles {$\mathcal{D}$} is a triangulated category. We will dedicate this section to prove such result and we will use some technical  following remarks, lemmas and propositions in which has a similar spirit as in \cite{BCP24}.


\begin{remark}\label{remark2}
For any morphism  $T \in \textup{Hom}_{s(\mathscr{Y},\Bbbk)}(X,Y)$ denote  by  $\overline{T}$ the equivalence class of 
 $T$ in $\textup{Hom}_{\mathcal{K}(\mathscr{Y},\Bbbk)}(X,Y)$, hence  for any $u \in \textup{Hom}_{\mathcal{K}(\mathscr{Y},\Bbbk)}(X,Y)$, there exists a distinguished triangle $(X,Y,\Omega_T,\overline{T},\overline{\iota}_Y,\overline{\pi}_{\llbracket X\rrbracket})$ in $\mathcal{D}$, where $u=\overline{T}$. Moreover, the family {$\mathcal{D}$} is closed under isomorphism.

\end{remark}

Since some blocks in $s(\mathscr{Y}, \Bbbk)$ may be empty, the zero object in $s(\mathscr{Y}, \Bbbk)$ corresponds to the matrix where all the blocks are empty, which will be denoted by $\mathbb{O}$. 
\begin{proposition}
${\Omega_{\textup{Id}_B}}\cong \mathbb{O}$ in $\mathcal{K}(\mathscr{Y},\Bbbk)$, for any $B \in s(\mathscr{Y}, \Bbbk)$. 
\end{proposition}
\begin{proof}
To show that ${\Omega_{\textup{Id}_B}}\cong \mathbb{O}$ is sufficiens see that  $\textup{Id}_{{\Omega_{\textup{Id}_B}}}\simeq \mathbf{0}$. To this end, consider the  matrix   
$L={\footnotesize\left(\begin{array}{cc}
    (\mathbf{0}_{B}^{B})_{u_{i+1}}^{v_{j+1}} & (\mathbf{0}_{B}^B)_{u_{i+1}}^{v_j}  \\
    (\textup{Id}_B)_{u_i}^{v_{j+1}} & (\mathbf{0}_B^B)_{u_i}^{v_j}
    \end{array}\right)_{{u_i},{v_j} \in \mathscr{Y}}}$, which  satisfies (i), (iii) and (iv) from definition of $\mathcal{K}$-matrix and the  equality $(\textup{Id}_{{\Omega_{\textup{Id}_B}}})_{u_i}^{v_j}=({\Omega_{\textup{Id}_B}}L)_{u_i}^{v_j}+(L{\Omega_{\textup{Id}_B}})_{u_i}^{v_j}$ (condition (ii)) is consequence of the following computations
    {\footnotesize{\begin{align*}
({\Omega_{\textup{Id}_B}}L)_{u_i}^{v_j}&=\sum_{w_k\in \mathscr{Y}}\left(\begin{array}{cc}
    -B_{u_{i+1}}^{w_k} & (\textup{Id}_B)_{u_{i+1}}^{w_k}  \\
    (\mathbf{0}_B^{B})_{u_i}^{w_k} & B_{u_i}^{w_k}
    \end{array}\right)\left(\begin{array}{cc}
    (\mathbf{0}_{B}^{B})_{w_k}^{v_{j+1}} & (\mathbf{0}_{B}^B)_{w_k}^{v_j}  \\ 
    (\textup{Id}_B)_{w_k}^{v_{j+1}} & (\mathbf{0}_B^B)_{w_k}^{v_j}
    \end{array}\right) =\left(\begin{array}{cc}
    (\textup{Id}_B)_{u_{i+1}}^{v_{j+1}} & (\mathbf{0}_{B}^B)_{u_{i+1}}^{v_j}  \\
    B_{u_i}^{v_{j+1}} & (\mathbf{0}_B^B)_{u_i}^{v_j}
        \end{array}\right),
    \end{align*}}}
    {\footnotesize{\begin{align*}
(L{\Omega_{\textup{Id}_B}})_{u_i}^{v_j}&=\sum_{w_k\in \mathscr{Y}}\left(\begin{array}{cc}
    (\mathbf{0}_{B}^{B})_{u_{i+1}}^{w_k}  & (\mathbf{0}_{B}^{B})_{u_{i+1}}^{w_k}   \\
    (\textup{Id}_B)_{u_i}^{w_k}  & (\mathbf{0}_B^B)_{u_i}^{w_k} 
    \end{array}\right)\left(\begin{array}{cc}
    -B_{w_k}^{v_{j+1}}  & (\textup{Id}_B)_{w_k}^{v_j}   \\
    (\mathbf{0}_B^{B})_{w_k}^{v_{j+1}} & B_{w_k}^{v_j}
    \end{array}\right)= \left(\begin{array}{cc}
    (\mathbf{0}_{B}^{B})_{u_{i+1}}^{v_{j+1}} & (\mathbf{0}_{B}^B)_{u_{i+1}}^{v_j}  \\
    -B_{u_i}^{v_{j+1}}& (\textup{Id}_B)_{u_i}^{v_j}
    \end{array}\right)
    \end{align*}}}for all  ${u_i},{v_j} \in \mathscr{Y}$. 
\end{proof}
\begin{corollary}\label{zerok}
For any $B \in s(\mathscr{Y},\Bbbk)$, the triangule  $(B,B,\mathbb{O},\textup{Id}_B,\mathbf{0},\mathbf{0} )$ is distinguished.
\end{corollary}
Let us continue with the rotation property for distinguished triangles.

\begin{proposition}
\label{prop:TR2}
If $(X,Y,Z,u,v,w)$ is a distinguished triangle, then $(Y,Z,\llbracket X \rrbracket,v,w,-\llbracket u\rrbracket)$ is a distinguished triangle.
\end{proposition}


\begin{proof}

Since the rotation property is compatible with isomorphisms of triangles, it is enough to prove  for a standard triangle $
    \xymatrix{
    B\ar[r]^-{T} &
    C\ar[r]^-{\iota_C} &
    {\Omega_T}\ar[r]^-{\pi _{\llbracket B \rrbracket}} &
    \llbracket B \rrbracket}
    $. In other words, we shall prove that
    $
    \xymatrix{
    C\ar[r]^-{\iota_C} &
    {\Omega_T}\ar[r]^-{\pi _{\llbracket B \rrbracket}} &
    \llbracket B \rrbracket\ar[r]^-{-\llbracket T \rrbracket} &
    \llbracket C \rrbracket}
    $ is a distinguished triangle. Now, consider the following diagram
     {\footnotesize{$$
    \xymatrix{
    C\ar[r]^{\iota_C}\ar[d]_{\textup{Id}_C} & {\Omega_T}\ar[r]^-{\iota_{{\Omega_T}}}\ar[d]_{\textup{Id}_{{\Omega_T}}} & {\Omega_{\iota_C}}\ar[r]^{\pi_{\llbracket C \rrbracket}}\ar[d]_{S}   & \llbracket C \rrbracket\ar[d]^{\textup{Id}_{\llbracket C \rrbracket}}\\
    C\ar[r]_{\iota_C} & {\Omega_T}\ar[r]_-{\pi _{\llbracket B \rrbracket}} & \llbracket B \rrbracket\ar[r]_{-\llbracket T \rrbracket}   & \llbracket C \rrbracket}
    $$}}where $S$ is the morphism given in Lemma \ref{tr2:lemma}. The commutativity in $\mathcal{K}(\mathscr{Y},\Bbbk)$ of the latter diagram follows to the fact that $\iota_{{\Omega_T}}S=\pi _{\llbracket B \rrbracket}$ and $\pi_{\llbracket C \rrbracket}+S\llbracket T \rrbracket={\Omega_{\iota_C}}L+ L\llbracket C \rrbracket$, where 
    $$
    L=\left(
\begin{array}{cc}
     (\mathbf{0}_{{C}}^{C})_{u_{i+1}}^{v_{j+1}} \\
       (\mathbf{0}_{B}^{C})_{u_{i+1}}^{v_{j+1}} \\ 
       (\textup{Id}_C)_{u_i}^{v_{j+1}}
\end{array}
\right)_{{u_i},{v_j}\in \mathscr{Y}}
    $$
is a $\mathcal{K}$-matrix. To show that $S$ is an isomorphism in $\mathcal{K}(\mathscr{Y},\Bbbk)$, it is enough to consider the morphism $R$ introduced in Lemma~\ref{tr2:lemma}, and then note that $RS=\textup{Id}_{\llbracket B \rrbracket}$ and
$\textup{Id}_{{\Omega_{\iota_C}}}-SR={L}{\Omega_{\iota_C}}+{\Omega_{\iota_C}}L$, where 
{\small{$$
L= \left(\begin{array}{cccc}
   (\mathbf{0}_{C}^{C})_{u_{i+1}}^{v_{j+1}}   & (\mathbf{0}_{C}^{B})_{u_{i+1}}^{v_{j+1}}   &   (\mathbf{0}_{C}^C)_{u_{i+1}}^{v_j}\\ 
(\mathbf{0}_{B}^{C})_{u_{i+1}}^{v_{j+1}}   & (\mathbf{0}_{B}^{B})_{u_{i+1}}^{v_{j+1}}   & (\mathbf{0}^{C}_{B})_{u_{i+1}}^{v_j} \\
    (\textup{Id}_{C})_{u_i}^{v_{j+1}}  & (\mathbf{0}_{C}^{B})_{u_i}^{v_{j+1}} & (\mathbf{0}_{C}^C)_{u_i}^{v_j}    
\end{array}\right)_{u_i,v_j\in \mathscr{Y}}
$$}}
is a $\mathcal{K}$-matrix.
\end{proof}

In light of Remark~\ref{remark2}, Corollary \ref{zerok}, and  propositions~\ref{prop:TR2} and  \ref{prop:TR3}, we can guarantee that $\mathcal{K}(\mathscr{Y},\Bbbk)$ is a pretriangulated category.
\begin{proposition}
\label{prop:TR3}
If $(X,Y,Z,u,v,w)$ and  $(X',Y',Z',u',v',w')$ are distinguished triangles, then for any  $f\in \textup{Hom}_{\mathcal{K}(\mathscr{Y},\Bbbk)}(X,X')$ and  $g\in\textup{Hom}_{\mathcal{K}(\mathscr{Y},\Bbbk)}(Y,Y')$ of morphisms such that $fu'=ug$, there exists a morphism $h\in\textup{Hom}_{\mathcal{K}(\mathscr{Y},\Bbbk)}(Z,Z')$ such that the following diagram commutes in  $\mathcal{K}(\mathscr{Y},\Bbbk)$
    $$
    \xymatrix{X \ar[d]_{ f}\ar[r]^{u} &{Y}\ar[d]_{g}\ar[r]^{v} & Z\ar@{-->}[d]_{\exists h}\ar[r]^{ w} &\llbracket X \rrbracket\ar[d]^{\llbracket f \rrbracket} \\
    X'\ar[r]_{u'} & Y'\ar[r]_{v'} & Z'\ar[r]_{w'} & \llbracket X' \rrbracket}
    $$
\end{proposition}


\begin{proof}

Again, it suffices to prove this proposition for standard triangles. By assumption we have a diagram
\begin{align}\label{diagram3.7}
\xymatrix{{B} \ar[d]_{ F}\ar[r]^-{ T} & {C}\ar[d]_-{G}\ar[r]^-{\iota_{C}} & {{\Omega_T }}\ar@{-->}[d]_-{H}\ar[r]^-{\pi _{\llbracket B \rrbracket}} &{\llbracket B \rrbracket}\ar[d]^-{\llbracket F \rrbracket} \\
{B'}\ar[r]_-{T'} & {C'}\ar[r]_-{ \iota_{C'} } & {{\Omega_{T'}} }\ar[r]_-{\pi_{\llbracket B' \rrbracket} } & {\llbracket B' \rrbracket}}
\end{align}
where the left square commutes in $\mathcal{K}(\mathscr{Y},\Bbbk)$. This implies that there exists a $\mathcal{K}$-matrix $L$ such that $FT'\simeq TG$. From property (i) of $\mathcal{K}$-matrix we can consider the matrix
    $$
    H =\left(\begin{array}{cc}
   \llbracket F \rrbracket_{u_i}^{v_j}  & L_{u_i}^{v_j}  \\ 
    (\mathbf{0}_{C}^{\llbracket B' \rrbracket})_{u_i}^{v_j} &  G_{u_i}^{v_j} 
\end{array}\right)_{{u_i},{v_j} \in \mathscr{Y}}=\left(\begin{array}{cc}
   F_{u_{i+1}}^{v_{j+1}}  & L_{u_{i+1}}^{v_j}   \\
    (\mathbf{0}_{C}^{B'})_{u_i}^{v_{j+1}} &  G_{u_i}^{v_j} 
\end{array}\right)_{{u_i},{v_j} \in \mathscr{Y}}.
    $$
Let us  show that $H\in\textup{Hom}_{s(\mathscr{Y},\Bbbk)}({\Omega_T},{\Omega_{T'}} )$. The properties (a) and (d) of the definition of morphism in $s(\mathscr{Y},\Bbbk)$,  is straightforward of the fact that $F$ and $G$ are morphisms in $s(\mathscr{Y},\Bbbk)$ and $L$ is a $\mathcal{K}$-matrix. 

To see that $H^{v_j}_{u_i}=0$  for all $v_j<u_i \in \mathscr{Y}$ (item (c) of morphism). First of all, note that
  $$H_{u_i}^{v_j} =\left(\begin{array}{cc}
   F_{u_{i+1}}^{v_{j+1}}  & L_{u_{i+1}}^{v_j}   \\
    (\mathbf{0}_{C}^{B'})_{u_i}^{v_{j+1}} &  G_{u_i}^{v_j} 
\end{array}\right)= \left(\begin{array}{cc}
   (\mathbf{0}_B^{B'})_{u_{i+1}}^{v_{j+1}}  & L_{u_{i+1}}^{v_j}   \\
    (\mathbf{0}_{C}^{B'})_{u_i}^{v_{j+1}} &  (\mathbf{0}_C^{C'})_{u_i}^{v_j} 
\end{array}\right)
    $$ 
because that $\llbracket F\rrbracket $ and $G$ are morphisms. It remains to be seen that the block $L_{u_{i+1}}^{v_j} = 0$. In this sense, we will proceed with the following case analysis: 

\begin{itemize}
\item If $i>j$ the block $
 L^{v_j}_{u_{i+1}}$ is zero, because that $i+1>j+1$ (by item \textup{(iii)} of $\mathcal{K}$-matrix).
\item If $u>v$ and $i=j$, the block $L^{v_i}_{u_{i+1}}$ is zero, because that $u>v$ and $i+1=i+1$ (by item \textup{(iii)} of $\mathcal{K}$-matrix).
\end{itemize}

For equality $\Omega_TH=H\Omega_{T'}$ (item (b) of morphism) is consequence of $FT'-TG=BL+LC'$, $BF=FB'$ and $CG=GC'$, because
{\small{\begin{align*}
({\Omega_T} H)_{u_i}^{v_j}&=
\sum_{w_k\in \mathscr{Y}}\left(\begin{array}{cc}
    -B_{u_{i+1}}^{w_k} & T_{u_{i+1}}^{w_k}  \\
    (\mathbf{0}^{B}_{C})_{u_{i}}^{w_k} &  C_{u_{i}}^{w_k}
\end{array}\right)\left(\begin{array}{cc}
   F_{w_k}^{v_{j+1}}  & L_{w_k}^{v_j}   \\
    (\mathbf{0}_{C}^{B'})_{w_k}^{v_{j+1}} &  G_{w_k}^{v_j} 
\end{array}\right)= \left(\begin{array}{cc}
    -(BF)_{u_{i+1}}^{v_{j+1}} & (TG-BL)_{u_{i+1}}^{v_j} \\ 
    (\mathbf{0}^{B'}_{C})_{u_i}^{v_{j+1}} &  (CG)_{u_i}^{v_j}
\end{array}\right),\\
(H{\Omega_{T'}})_{u_i}^{v_j} &= 
\sum_{w_k\in \mathscr{Y}}\left(\begin{array}{cc}
   F^{w_k}_{u_{i+1}}  & L^{w_k}_{u_{i+1}}  \\ 
    (\mathbf{0}_{C}^{B'})^{w_k}_{u_i}  &  G^{w_k}_{u_i}  
\end{array}\right)\left(\begin{array}{cc}
    (-B')_{w_k}^{v_{j+1}} & (T')_{w_k}^{v_j}  \\ (\mathbf{0}^{B'}_{C'})_{w_k}^{v_{j+1}} &  (C')_{w_k}^{v_j}
\end{array}\right)= \left(\begin{array}{cc}
    -(FB')_{u_{i+1}}^{v_{j+1}} & (FT' + LB')_{u_{i+1}}^{v_j}   \\ 
  (\mathbf{0}^{B'}_{C})_{u_i}^{v_{j+1}}  &  (GC')_{u_i}^{v_j} 
\end{array}\right)
\end{align*}}} for all $u_i,v_j \in \mathscr{Y}$.

The commutativity of the following  diagram \eqref{diagram3.7}
 {\small{
  \begin{align*}
    ( G\iota_{C'})_{u_i}^{v_j} & =\sum_{w_k\in \mathscr{Y}} G_{u_i}^{w_k}\left(\begin{matrix}
    (\mathbf{0}^{B'}_{C'})_{w_k}^{v_{j+1}}  \ \  (\textup{Id}_{C'})_{w_k}^{v_j}
\end{matrix}\right)
=\sum_{w_k \in \mathscr{Y}}\left(\begin{matrix}
            (\mathbf{0}^{B}_{C } )_{u_i}^{w_k}  \ \  (\textup{Id}_{C})_{u_i}^{w_k}
\end{matrix}\right)\left(\begin{array}{cc}
   F_{w_k}^{v_{j+1}}  \ \  L_{w_k}^{v_j}  \\ 
    (\mathbf{0}_{C}^{B'})_{w_k}^{v_{j+1}} \ \   G_{w_k}^{v_j}
\end{array}\right)=( \iota_{C}H)_{u_i}^{v_j},\\
(\pi _{\llbracket B \rrbracket} \llbracket F \rrbracket)_{u_i}^{v_j} & = \sum_{w_k\in \mathscr{Y}}\left(\begin{matrix}
             (\textup{Id}_{B})_{u_{i+1}}^{w_k}\\ 
              (\mathbf{0}^{B}_{C})_{u_i}^{w_k}
\end{matrix}\right)F_{w_k}^{v_{j+1}} =\left(\begin{matrix}
       F_{u_{i+1}}^{v_{j+1}} \\
              (\mathbf{0}_{C}^{B'})^{v_{j+1}}_{u_i}
\end{matrix}\right) = \sum_{w_k \in \mathscr{Y}}\left(\begin{array}{cc}
   F^{w_k}_{u_{i+1}}  \  \ L^{w_k}_{u_{i+1}}   \\
    (\mathbf{0}_{C}^{B'})^{w_k}_{u_i}  \ \   G^{w_k}_{u_i}  
\end{array}\right)\left(\begin{matrix}
            (\textup{Id}_{B'})_{w_k}^{v_{j+1}} \\ 
             (\mathbf{0}^{B'}_{C'})_{w_k}^{v_{j+1}}
\end{matrix}\right)= (H\pi_{\llbracket B' \rrbracket})_{u_i}^{v_j}
\end{align*}}}
for all $u_i,v_j\in\mathscr{Y}$.

\end{proof}


In order to prove that $\mathcal{K}(\mathscr{Y},\Bbbk)$ is a triangulated category, it remains to show the octahedral axiom.


\begin{proposition}
\label{prop:Octae}
  \ For any $(X, Y, X', u, u', v')$, $(Y, Z, Z', v,w, w')$ and $(X, Z, Y', uv, 
   p,q)$ distinguished triangles, there exists a distinguished triangle 
   $(X',Y',Z',f,g,w'u')$  making the following diagram commutative in  $\mathcal{K}(\mathscr{Y},\Bbbk)$
    {\footnotesize{$$\xymatrix{
    X\ar[rr]^-{u}\ar[d]_-{\textup{Id}_X}   &&Y\ar[rr]^-{u'}\ar[d]_-{v}    &&X'\ar@{-->}[d]_-{f}\ar[rr]^<<<<<{v'} &    &\llbracket X \rrbracket\ar[d]^-{\textup{Id}_{\llbracket X \rrbracket}}\\
X\ar[rr]^-{uv}\ar[d]_-{u}  & &Z\ar[rr]^-{p}\ar[d]_{\textup{Id}_Z}    &&Y'{ \ar@{-->}[d]_-{g}}\ar[rr]^-{q}    & &\llbracket X \rrbracket\ar[d]^-{\llbracket u\rrbracket} \\   
Y\ar[rr]^-{v}  & &Z\ar[rr]^-{w}    &&Z'\ar[d]_-{w'u'}\ar[rr]^-{w'}    & &\llbracket Y \rrbracket\ar[lld]^-{\llbracket u'\rrbracket} \\ 
&& && \llbracket X' \rrbracket&&}
    $$}} 
\end{proposition}


\begin{proof}

Again, it suffices to prove the Octahedral axiom for standard triangles. First of all, let us show that the following diagram commutes in $\mathcal{K}(\mathscr{Y},\Bbbk)$

  {\footnotesize{$$\xymatrix{B\ar[rr]^{S}\ar[d]_{\textup{Id}_B}   &&C\ar[rr]^{\iota_C}\ar[d]_{T}    &&{\Omega_{S}} \ar@{-->}[d]_{\exists F}\ar[rr]^-{\pi _{\llbracket B \rrbracket}} &    & \llbracket B \rrbracket\ar[d]^{\textup{Id}_{\llbracket B \rrbracket}}\\
B\ar[rr]^{ST}\ar[d]_{S}  & &D\ar[rr]^-{\overline{\iota_D}}\ar[d]_-{\textup{Id}_D}    &&{\Omega_{ST}} \ar@{-->}[d]_{\exists G}\ar[rr]^-{\overline{\pi _{\llbracket B \rrbracket}}}    & & \llbracket B \rrbracket\ar[d]^{\llbracket S \rrbracket} \\   
C\ar[rr]_{T}  & &D\ar[rr]_-{{\iota_D}}    &&{\Omega_T} \ar[rr]_-{\pi_{\llbracket C \rrbracket}}    & &\llbracket C \rrbracket
}$$}}where, to avoid confusion,  $\overline{\iota_D}$ and  $\overline{\pi _{\llbracket B \rrbracket}}$ denote the morphism $\iota_D:D\longrightarrow {\Omega_{ST}}$ and  $\pi _{\llbracket B \rrbracket}:{\Omega_{ST}}\longrightarrow \llbracket B \rrbracket$ defined in Lemma~\ref{lem:CT} and, $F\in\textup{Hom}_{s(\mathscr{Y},\Bbbk)}({\Omega_S}, {\Omega_{ST}})$ and $G\in\textup{Hom}_{s(\mathscr{Y},\Bbbk)}({\Omega_{ST}}, {\Omega_{T}})$ are the morphisms given in Lemma \ref{tr6:lemma}, specifically, 
     {\small{  $$
F = \left(\begin{array}{cc}
   (\textup{Id}_{B})_{u_{i+1}}^{v_{j+1}} & (\mathbf{0}_{B}^D)_{u_{i+1}}^{v_j}  \\ 
    (\mathbf{0}_C^{B})_{u_i}^{v_{j+1}} & T_{u_i}^{v_j}  
\end{array}\right)_{{u_i},{v_j} \in \mathscr{Y}}
\ \ \ \ \ \text{and }\ \ \ \ \ G = \left(\begin{array}{cc}
   S_{u_{i+1}}^{v_{j+1}}  & (\mathbf{0}_B^D)_{u_{i+1}}^{v_j}  \\ 
    (\mathbf{0}_D^{C})_{u_i}^{v_{j+1}} & (\textup{Id}_D)_{u_i}^{v_j}  
\end{array}\right)_{{u_i},{v_j} \in \mathscr{Y}}.$$}} 

The commutativity is straightforward by multiplication of matrices, for instance,
{\small{\begin{align*}
    (\iota_C F)_{u_i}^{v_j} &=\sum_{w_k\in \mathscr{Y}}\left(\begin{array}{cc}
   (\mathbf{0}_C^{B})_{u_i}^{w_k} & (\text{Id}_{C})_{u_i}^{w_k}  \\
\end{array}\right)\left(\begin{array}{cc}
   (\textup{Id}_{B})_{w_k}^{v_{j+1}} & (\mathbf{0}_B^D)_{w_k}^{v_j}  \\ 
    (\mathbf{0}_C^{B})_{w_k}^{v_{j+1}} & T_{w_k}^{v_j}  
\end{array}\right)=\sum_{w_k\in \mathscr{Y}}T_{u_i}^{w_k}\left(\begin{array}{cc}
   (\mathbf{0}_D^{B})_{w_k}^{v_{j+1}} & (\textup{Id}_D)_{w_k}^{v_j}
\end{array}\right)= (T\overline{\iota_{D}})_{u_i}^{v_j},\\
    (\overline{\iota_D}G)_{u_i}^{v_j} &=\sum_{w_k\in \mathscr{Y}}\left(\begin{array}{cc}
   (\mathbf{0}_D^{B})_{u_i}^{w_k} & (\text{Id}_{D})_{u_i}^{w_k}  \\
\end{array}\right)\left(\begin{array}{cc}
   S_{w_k}^{v_{j+1}}  & (\mathbf{0}_B^D)_{w_k}^{v_j}  \\ 
    (\mathbf{0}_D^{C})_{w_k}^{v_{j+1}} & (\textup{Id}_D)_{w_k}^{v_j}  
\end{array}\right)= \left(\begin{array}{cc}
   (\mathbf{0}_D^{C})_{u_i}^{v_{j+1}} & (\text{Id}_{D})_{u_i}^{v_j}  \\
\end{array}\right)=({\iota_{D}})_{u_i}^{v_j}
\end{align*}}}for all $u_i,v_j \in \mathscr{Y}$.
 
By Lemma \ref{tr6:lemma} we have that
	$$
    \Lambda=\left(\begin{array}{ccccc}
     (\mathbf{0}_{C}^{B})_{u_{i+1}}^{v_{j+2}}& (\textup{Id}_{C})_{u_{i+1}}^{v_{j+1}} &
     (\mathbf{0}_{C}^{B})_{u_{i+1}}^{v_{j+1}}& (\mathbf{0}_{C}^{D})_{u_{i+1}}^{v_j}\\ 
     (\mathbf{0}_{D}^{B})_{u_i}^{v_{j+2}}& (\mathbf{0}_D^{C})_{u_i}^{v_{j+1}} &
     (\mathbf{0}_{D}^{B})_{u_i}^{v_{j+1}} & (\textup{Id}_D)_{u_i}^{v_j} 
\end{array}\right)_{{u_i},{v_j} \in \mathscr{Y}}\in\textup{Hom}_{s(\mathscr{Y},\Bbbk)}({\Omega_T},{\Omega_F}).
    $$
    Let us show that $\Lambda$ is an isomorphism in $\mathcal{K}(\mathscr{Y},\Bbbk)$ such that the following diagram commutes in $\mathcal{K}(\mathscr{Y},\Bbbk)$
{\footnotesize{$$\xymatrix@C=15pt{	
	{{\Omega_S}}\ar[rrr]^-{F}\ar[d]_-{\textup{Id}_{ {\Omega_S}}}	&&&	{ {\Omega_{ST}}}\ar[rrr]^-{ G}\ar[d]_-{\textup{Id}_{{\Omega_{ST}}} }		&&&	
 {{\Omega_{T}} }\ar[rrr]^-{ \pi_{\llbracket C \rrbracket}\iota_{\llbracket C \rrbracket}}\ar@{-->}[d]_-{\exists \Lambda}	&	&&{{\llbracket\Omega_{ S }\rrbracket} }\ar[d]^-{\textup{Id}_{\llbracket{\Omega_{ S }}\rrbracket}}		\\
	 {{\Omega_S}}\ar[rrr]_{ F}&&&{{\Omega_{ST} }}\ar[rrr]_-{ \iota_{{\Omega_{ST}}}} &&&
 {{\Omega_{F}} }\ar[rrr]_-{\pi_{{{{\llbracket\Omega_{ S }\rrbracket} }}} }	&&&{{\llbracket\Omega_{ S }\rrbracket} } }$$}}

The commutativity is straightforward by multiplication of matrices, for instance 
{\small{\begin{align*}
(&\iota_{{\Omega_{ST}} }-G\Lambda)_{u_i}^{v_j} \\
&=\left(\begin{array}{ccccc}
     (\mathbf{0}_{B}^{B})_{u_{i+1}}^{v_{j+2}}& (\mathbf{0}_B^C)_{u_{i+1}}^{v_{j+1}} &
     (\textup{Id}_B)_{u_{i+1}}^{v_{j+1}}& (\mathbf{0}_{B}^{D})_{u_{i+1}}^{v_j}\\ 
     (\mathbf{0}_{D}^{B})_{u_i}^{v_{j+2}}& (\mathbf{0}_D^{C})_{u_i}^{v_{j+1}} &
     (\mathbf{0}_{D}^{B})_{u_i}^{v_{j+1}} & (\textup{Id}_D)_{u_i}^{v_j} 
\end{array}\right)\!\!  -\!\!\left(\begin{array}{ccccc}
     (\mathbf{0}_{B}^{B})_{u_{i+1}}^{v_{j+2}}& S_{u_{i+1}}^{v_{j+1}} &
     (\mathbf{0}_B^B)_{u_{i+1}}^{v_{j+1}}& (\mathbf{0}_B^D)_{u_{i+1}}^{v_j}\\ 
     (\mathbf{0}_{D}^{B})_{u_i}^{v_{j+2}}& (\mathbf{0}_D^{C})_{u_i}^{v_{j+1}} &
     (\mathbf{0}_{D}^{B})_{u_i}^{v_{j+1}} & (\textup{Id}_D)_{u_i}^{v_j} 
\end{array}\right)\\
&= \left(\begin{array}{ccccc}
     (\mathbf{0}_{B}^{B})_{u_{i+1}}^{v_{j+2}}& -S_{u_{i+1}}^{v_{j+1}} &
     (\textup{Id}_B)_{u_{i+1}}^{v_{j+1}}& (\mathbf{0}_B^D)_{u_{i+1}}^{v_j}\\ 
     (\mathbf{0}_{D}^{B})_{u_i}^{v_{j+2}}& (\mathbf{0}_D^{C})_{u_i}^{v_{j+1}} &
     (\mathbf{0}_{D}^{B})_{u_i}^{v_{j+1}} & (\mathbf{0}_D^D)_{u_i}^{v_j} 
\end{array}\right)\\
&=\sum_{w_k \in \mathscr{Y}}\!\!\left(\begin{array}{cc}
   -B_{u_{i+1}}^{w_k}   & (ST)_{u_{i+1}}^{w_k}   \\ 
    (\mathbf{0}_D^{B})_{u_i}^{w_k} & D_{u_i}^{w_k}   
\end{array}\right)\left(\begin{array}{ccccc}
     (\textup{Id}_B)_{w_k}^{v_{j+2}}& (\mathbf{0}_B^C)_{w_k}^{v_{j+1}} &
     (\mathbf{0}_B^B)_{w_k}^{v_{j+1}}& (\mathbf{0}_{B}^{D})_{w_k}^{v_j}\\ 
     (\mathbf{0}_{D}^{B})_{w_k}^{v_{j+2}}& (\mathbf{0}_D^{C})_{w_k}^{v_{j+1}} &
     (\mathbf{0}_{D}^{B})_{w_k}^{v_{j+1}} & (\mathbf{0}_D^D)_{w_k}^{v_j} 
\end{array}\right)\\
&+\sum_{w_k \in \mathscr{Y}}\!\!\left(\begin{array}{ccccc}
     (\textup{Id}_B)_{u_{i+1}}^{w_k}& (\mathbf{0}_B^C)_{u_{i+1}}^{w_k} &
     (\mathbf{0}_B^B)_{u_{i+1}}^{w_k}& (\mathbf{0}_{B}^{D})_{u_{i+1}}^{w_k}\\ 
     (\mathbf{0}_{D}^{B})_{u_i}^{w_{k}}& (\mathbf{0}_D^{C})_{u_i}^{w_{k}}&
     (\mathbf{0}_{D}^{B})_{u_i}^{w_{k}} & (\mathbf{0}_D^D)_{u_i}^{w_{k}} 
\end{array}\right) \!\! \left(\begin{array}{cccc}
        B_{w_k}^{v_{j+2}}  & -S_{w_k}^{v_{j+1}}& (\textup{Id}_{B})_{w_k}^{v_{j+1}} & (\mathbf{0}_B^D)_{w_k}^{v_j}  \\ 
          (\mathbf{0}_{C}^{B})_{w_k}^{v_{j+2}}  & -C_{w_k}^{v_{j+1}} & (\mathbf{0}_{C}^{B})_{w_k}^{v_{j+1}} & T_{w_k}^{v_j} \\ 
             (\mathbf{0}_{B}^{B})_{w_k}^{v_{j+2}} &(\mathbf{0}_{B}^{C})_{w_k}^{v_{j+1}} & -B_{w_k}^{v_{j+1}}& (ST)_{w_k}^{v_j}  \\ 
                 (\mathbf{0}_D^{B})_{w_k}^{v_{j+2}}& (\mathbf{0}_D^{C})_{w_k}^{v_{j+1}}&(\mathbf{0}_D^{B})_{w_k}^{v_{j+1}} &D_{w_k}^{v_j}  \\
     \end{array}\right)\\
     &=( {\Omega_{ST}}L +L{\Omega_{F}})_{u_i}^{v_j}
\end{align*}}}
for all $u_i,v_j \in \mathscr{Y}$, where 
    \begin{align*}
    L=\left(\begin{array}{ccccc}
     (\textup{Id}_B)_{u_{i+1}}^{v_{j+2}}& (\mathbf{0}_B^C)_{u_{i+1}}^{v_{j+1}} &
     (\mathbf{0}_B^B)_{u_{i+1}}^{v_{j+1}}& (\mathbf{0}_{B}^{D})_{u_{i+1}}^{v_j}\\ 
     (\mathbf{0}_{D}^{B})_{u_i}^{v_{j+2}}& (\mathbf{0}_D^{C})_{u_i}^{v_{j+1}} &
     (\mathbf{0}_{D}^{B})_{u_i}^{v_{j+1}} & (\mathbf{0}_D^D)_{u_i}^{v_j} 
\end{array}\right)_{u_i,v_j \in \mathscr{Y}}
\end{align*}
is a $\mathcal{K}$-matrix. Moreover,
{\small{\begin{align*}
    (\Lambda\pi_{\llbracket\Omega_{S}\rrbracket})^{v_j}_{u_i}&= \sum_{w_k\in \mathscr{Y}}\left(\begin{array}{ccccc}
     (\mathbf{0}_{C}^{B})_{u_{i+1}}^{w_k} & (\textup{Id}_{C})_{u_{i+1}}^{w_k} &
     (\mathbf{0}_{C}^{B})_{u_{i+1}}^{w_k} &(\mathbf{0}_{C}^{D})_{u_{i+1}}^{w_k} \\ 
     (\mathbf{0}_{D}^{B})_{u_i}^{w_k} & (\mathbf{0}_D^{C})_{u_i}^{w_k} &
     (\mathbf{0}_{D}^{B})_{u_i}^{w_k} & (\textup{Id}_D)_{u_i}^{w_k} 
\end{array}\right)\left(\begin{array}{cc}
    (\textup{Id}_{B})_{w_k}^{v_{j+2}} & (\mathbf{0}_{B}^{C})_{w_k}^{v_{j+1}}  \\
    (\mathbf{0}^{B}_{C})_{w_k}^{v_{j+2}} & (\textup{Id}_{C})_{w_k}^{v_{j+1}}\\
      (\mathbf{0}_{B}^{B})_{w_k}^{v_{j+2}} & (\mathbf{0}_{B}^{C})_{w_k}^{v_{j+1}}  \\
    (\mathbf{0}^{B}_D)_{w_k}^{v_{j+2}} & (\mathbf{0}_D^{C})_{w_k}^{v_{j+1}}\\
\end{array}\right)\\
&= \left(\begin{array}{cc}
    (\mathbf{0}_{C}^{B})_{u_{i+1}}^{v_{j+2}} & (\textup{Id}_{C})_{u_{i+1}}^{v_{j+1}}  \\
    (\mathbf{0}^{B}_D)_{u_{i}}^{v_{j+2}} & (\mathbf{0}_{D}^{C})_{u_{i}}^{v_{j+1}}\\
    \end{array}\right)= \sum_{w_k\in \mathscr{Y}}\left(\begin{array}{cc}
   (\textup{Id}_{C})_{u_{i+1}}^{w_k} \\
    (\mathbf{0}^{C}_D)_{u_{i+1}}^{w_k} 
\end{array}\right)\left(\begin{array}{cc}
  (\mathbf{0}^{B}_C)^{v_{j+2}}_{w_k} & (\textup{Id}_{C})_{w_k}^{v_{j+1}}  
    \end{array}\right)\\
    & = (\pi_{\llbracket C \rrbracket} \iota_{\llbracket C \rrbracket})_{u_i}^{v_j} \\
\end{align*}}}for all $u_i,v_j \in \mathscr{Y}$. Therefore, the proof is done, since $\textup{Id}_{{\Omega_S}}$ and $\textup{Id}_{{\Omega_{ST}}}$ are isomorphism in $\mathcal{K}(\mathscr{Y},\Bbbk)$ and five lemma for pretriangulated category.

\end{proof}


From Remark~\ref{remark2}, Corollary\ref{zerok} and propositions~\ref{prop:TR2}, \ref{prop:TR3} and \ref{prop:Octae} we can conclude the main result of this section.


\begin{theorem}\label{main 1}
    The category $\mathcal{K}(\mathscr{Y},\Bbbk)$ with the autofunctor $\llbracket - \rrbracket:\mathcal{K}(\mathscr{Y},\Bbbk)\longrightarrow \mathcal{K}(\mathscr{Y},\Bbbk)$ and the family of distinguished  triangles {$\mathcal{D}$},  is a triangulated category.
\end{theorem}


\section{Relationship with the homotopy category of gentle algebras}
\label{sec:DG}

In this section, we aim to construct and exhibit a triangulated functor from the homotopy category of projective modules over gentle algebras to a particular quotient of a Bondarenko's category. This functor will facilitate a deeper understanding of the relationships between these categories and their structural properties. To this end, we review some of the standard facts on representations theory of associative algebras, see \cite{Rin84,Sch14} for more details. 

In this section we consider certain finite dimensional quotients of path algebras. So, let $Q$ denote a finite quiver with set of vertices $Q_0$ and set of arrows $Q_1$. Suppose that $\Bbbk Q$ is the corresponding path algebra over the algebraically closed field $\Bbbk$ and let $I$ be an ideal of $\Bbbk Q$ such that $J^n\subseteq I \subseteq J^2$ for some integer $n\geq 2$ (i.e. $I$ is \emph{admissible}), where $J$ is the two-sided ideal generated by the arrows. Throughout this section, $A$ will denote an algebra of the form $\Bbbk Q / I=\Bbbk(Q,I)$ and $A$-mod the category of finitely generated left $A$-modules. 

We will denote  by $e_i$ the trivial path (of length 0) at vertex $i\in Q_0$ and by $P_i = Ae_i$ the corresponding indecomposable projective $A$-module. Let us denote by $\mathbf{Pa}$ the set of all paths of $A$, that is, all paths of $Q$ that are outside $I$, while $\mathbf{Pa}_{\geq l}$ will denote the subset of $\mathbf{Pa}$ of all paths of length greater than  or equal to a fixed non-negative integer $l$. Since, each element of $A$ is uniquely represented by a linear combination of paths in $\mathbf{Pa}$, we can assume that $\mathbf{Pa}$ forms a basis for $A$. And, if $w$ is a path, $s(w)$ denotes its \emph{source}, $t(w)$ denotes its \emph{target}, and $l(w)$ denotes its length.

Another special subset in $\mathbf{Pa}$, is the set called \emph{maximal paths} and denoted by $\mathbf{M}$, where a path $w$ in $A$ is considered \emph{maximal} if, for all arrows $a,b \in Q_1$, we have that $aw$ and $wb$ are zero in $A$. Furthermore, a nontrivial path $w$ in $Q$ belongs to 
$\mathbf{Pa}$ if and only if it is a sub-path of a maximal path, denoted by $\widetilde{w}$, that is not in $I$ (that is, an element of $\mathbf{M}$). This maximal path has the form $\widetilde{w} = \widehat{w}w\overline{w}$, where $\widehat{w}$ and $\overline{w} \in \mathbf{Pa}$.


\vspace{0,2cm}

For now on, we will assume that $A$ is a \emph{gentle algebra} (see \cite{AS87} for more details), i.e., an algebra $A=\Bbbk(Q,I)$ satisfying the following conditions:
\begin{enumerate}[(i)]
\item[\textup{(i)}] 
Each vertex in $Q$ is the source of at most two arrows and the target of at most two arrows.

\item[\textup{(ii)}] For any arrow $\alpha\in Q_1$ there is at most one arrow $\beta\in Q_1$ (respectively, $\gamma\in Q_1$) such that $\alpha\beta\not\in I$ (respectively, $\gamma\alpha\not\in I$).

\item[\textup{(iii)}] For any arrow $\alpha\in Q_1$ there is at most one arrow $\beta\in Q_1$ (respectively, $\gamma\in Q_1$) such that $\alpha\beta\in I$ (respectively, $\gamma\alpha\in I$).

\item[\textup{(iv)}] The ideal $I$ is generated by paths of length  two.
\end{enumerate}


\begin{example}
\label{ex:gentle}

We will illustrate with examples the constructions in this section with the following gentle algebras     
\begin{enumerate}[(i)]
\item 
$A_1=\Bbbk(Q,I)$ given by the quiver
	$
	\xymatrix{Q:\ \ & 1\ar@(ul,dl)_{x} \ar[r]^{a}  &2\ar@(ru,rd)[]^{y}}$ 
with the relations $I=\langle x^2 , y^2 \rangle$.

\item $A_2=\Bbbk Q$ given by the quiver
	$
	\xymatrix{
	Q:\ \ 1\ar@<2pt>[r]^{a}\ar@<-2pt>[r]_{b} & 2
	}
	$

\end{enumerate}

\end{example}


\vspace{0,2cm}

We denote by $\mathbf{D}(A)$ (respectively, $\mathbf{D}^b(A)$) the derived category of $A$-mod (respectively, the derived category of bounded complexes of $A$-mod), and by $\mathbf{C}^b(\textup{proj} A)$ the category of bounded complexes of projectives in $A\textup{-mod}$. Similarly, $\mathbf{K}^b(\textup{proj} A)$ denotes the corresponding homotopy category to $\mathbf{C}^b(\textup{proj} A)$.


\subsection{The functor}

In this subsection we will study a functor from category $\mathbf{K}^b(\textup{proj}A)$ to category $\mathcal{K}(\mathscr{Y}(A),\Bbbk)$, where $\mathscr{Y}(A)$ is the poset with involution (introduced in \cite[Section 3]{BM03}) defined of the following way: for each $m\in \mathbf{M}$ we define the poset $$\mathcal{Y}_m=\left\{e_{s(m)}<u_1<u_1u_2<\cdots < u_1u_2\cdots u_n\right\},$$ where $m=u_1u_2\cdots u_n$ and each $u_i\in Q_1$ is an arrow; note that $\mathcal{Y}_m$ is ordered by its length. Now, we assume on $\mathbf{M}$ a fixed linear order and we consider the disjoint union $\mathcal{Y}:=\bigcup_{m \in \mathbf{M}} \mathcal{Y}_m$. The involution $\sigma$ on $\mathcal{Y}$ is defined of the following way, $\sigma(u)=v$ if and only if $t(u)=t(v)$. Since there are no more than two paths
$u$, $v$ such that $t(u)=t(v)$ (see \cite[Proposition 2]{BM03}), in the case there is only one, $u$, we write that
$\sigma(u)=u$ and, when there are two, $u$, $v$, we let $\sigma$ interchange them. Similarly to Section \ref{sec:Bond}, the set
 	$$
    \mathscr{Y}(A):=\mathcal{Y} \times \mathbb{Z}=\left(\bigcup_{m \in \mathbf{M}} \mathcal{Y}_m\right) \times \mathbb{Z},
    $$ 
    is a poset ordered anti-lexicographically, that is
\begin{center}
    $[u,i]<[v,j]$ if and only if $i<j$ or ($i=j$ and $\widetilde{u}<\widetilde{v}$) or ($i=j$, $\widetilde{u}=\widetilde{v}$ and $l(u)<l(v)$),
\end{center}
with involution $\sigma$ on $\mathscr{Y}(A)$ is given by
 \begin{center}
 $\sigma([u,i])=[v,j]$ if and only if $i=j$ and $t(u)=t(v)$.    
 \end{center} 

  It should be noticed that it is possible that a trivial path $e_r$ ($r\in Q_0$) belongs to two different maximal paths. If this happens, the two occurrences of $e_r$ must be regarded as different. The example below illustrates this.
  

\begin{example}
\label{ex:Y(A)}
Let us consider the gentle algebras from Example~\ref{ex:gentle}.

$$
{\small \begin{array}{|c||c|c|c|c|}
\hline
 & \mathbf{Pa} & \mathbf{M} & \textup{Poset} & \textup{Involution} \\
             \hline\hline
A_1 & \left\{e_1,x,a,y,xa,ay,xay\right\} & \left\{xay\right\} & \left\{e_1<x<xa<xay\right\}\times \mathbb{Z}
 & \begin{cases}
        \sigma([e_1,j])=[x,j] \\
        \sigma([xa,j])=[xay,j]
     \end{cases}\\
 &  &  &  & \\
  A_2 & \left\{e_1,e_2,a,b\right\} &\left\{a,b\right\} & \left\{e_{s(a)}<a<e_{s(b)}<b\right\}\times \mathbb{Z} & \begin{cases}
        \sigma([e_{s(a)},j])=[e_{s(b)},j] \\
        \sigma([a,j])=[b,j]
     \end{cases} \\
\hline

\end{array}}
$$

\end{example}

 
\vspace{,3cm}

Let us define a functor between $\Bbbk$-categories $\mathbf{F}:\mathbf{C}^b(\textup{proj} A) \longrightarrow s(\mathscr{Y}(A),\Bbbk)$, which is an adapted version of the functor constructed in \cite[Section 3]{BM03}.

\vspace{0,5cm}

We start with a bounded complex $\textbf{P}^\bullet \in \mathbf{C}^b(\textup{proj} A)$ of length $m$, that is, a complex $\textbf{P}^\bullet$ of the form
    $$
    \xymatrix@C=10,6pt{\cdots\ar[rr] & & 0\ar[rr]   && \textbf{P}^n\ar[rr]^-{\partial^n}&& \textbf{P}^{n+1}\ar[rr]^-{\partial^{n+1}}&    & \cdots\cdots\ar[rr] &  & \textbf{P}^{n+m-1} \ \ar[rr]^-{\partial^{n+m-1}}   & & \textbf{P}^{n+m}\ar[rr] &  & 0\ar[rr] & & \cdots  }
    $$
with $n,m \in \mathbb{Z}$.

Since each projective $\textbf{P}^j$ is the finite direct sum of indecomposable projective, we can write the complex $\textbf{P}^\bullet$ as
    $$
    \xymatrix@C=10,2pt{
    \cdots\ar[rr] & & 0\ar[rr]   &&  \bigoplus\limits_{i=1}^t P_i^{d_{i, n}}\ar[rr]^-{\partial^n}&& \cdots\cdots\ar[rr] &  &\bigoplus\limits_{i=1}^t P_i^{d_{i, n+m-1}} \ar[rr]^-{\partial^{n+m-1}}   & & \bigoplus\limits_{i=1}^t P_i^{d_{i, n+m}}\ar[rr] &  & 0\ar[rr] & & \cdots  }
    $$
where $d_{i,j}$ denote the number of times that appears the projective $P_i$ in the place $j$. Moreover, each differential $\partial^j:\textbf{P}^j\longrightarrow\textbf{P}^{j+1}$ in the complex $\textbf{P}^\bullet$ is given by a block matrix of size $\sum\limits_{i\in Q_0} d_{i,j}\times \sum\limits_{{i\in Q_0}}d_{i,j+1}$, where each block matrix corresponds to a morphism $ P^{d_{r,j}}_r\longrightarrow P^{d_{s,j+1}}_s$.

As it is well known, the vector space $\textup{Hom}(P_r,P_s)$ has a basis consisting of all homomorphism $p(w):P_r\longrightarrow P_s$ (defined by $u\longmapsto v=uw$) with $w \in \mathbf{Pa}$, where $s(w)=r$ and $t(w)=s$. It follows that any morphism from $P_r$ to $P_s$ is associated to a linear combination of these $p(w)$. Here, we are also considering trivial paths.

Similarly to \cite[Section 3]{BM03} the complex $\textbf{P}^\bullet$ is represented by a block matrix, which is determined by the sequence of morphisms $\partial^j$, $j=n,n+1,\dots,n+m-1$ (and vice versa), where  each $\partial^j$ is given by a block matrix $\mathbf{A}= \left(\mathbf{A}^{s,j+1}_{r,j}\right)$, which depends on the ``multiplicities'' of the morphisms $p(w)$ in $\partial^j$. Precisely, each $\partial^j$ is represented by a formal sum: 
 \begin{align}\label{form:diff}
 \partial^j:\sum_{w \in \mathbf{Pa}}p(w)\mathbf{A}_{w,j},
\end{align}where $\mathbf{A}_{w,j}$ denotes the block that expresses the ``multiplicity'' of the morphism $p(w)$ at $\partial^j$. Allow us to explain this as follows. Fixed the place $j$ of the complex $\textbf{P}^\bullet$, the component of $\partial^j$ going from $P^{d_{r,j}}_r$ to $P^{d_{s,j+1}}_s$ is represented by a matrix (a block)
    $$
    \mathbf{A}^{s,j+1}_{r,j} \in \textup{Mat}\left(d_{r,j}\times d_{s,j+1};\Bbbk(\langle p(w_1),\dots,p(w_l)\rangle)\right),
    $$
where the $w_i$'s are parallel paths of $A$ from $r$ to $s$ and $ \Bbbk(\langle p(w_1),\dots , p(w_l)\rangle)$ is the $\Bbbk$-vector space with basis $\{p(w_1),\dots , p(w_l)\}$. It is then clear that $\mathbf{A}^{s,j+1}_{r,j}$ can be written uniquely as
    $$ \mathbf{A}^{s,j+1}_{r,j}=\sum_{i=1}^l p(w_i)\mathbf{A}_{w_i,j},
    $$
with $\mathbf{A}_{w_i,j} \in \textup{Mat}\left(d_{r,j}\times d_{s,j+1};\Bbbk\right)$. It is important to note that our convention is that, in the matrix representation of $\partial^j: \mathbf{P}^j\longrightarrow \mathbf{P}^{j+1}$, the indecomposable projectives summands in $\mathbf{P}^j$ correspond to rows, while those in $\mathbf{P}^{j+1}$ correspond to columns. 

Therefore, for each complex $\textbf{P}^\bullet \in \mathbf{C}^b(\textup{proj} A)$ we define the $({[u,j]},{[v,j+1]} )$th block by
    $$ 
    \mathbf{F}(\mathbf{P}^{\bullet} )^{[v,j+1]}_{[u,j]}:=\mathbf{A}_{w,j}
    $$
for all $j\in\mathbb{Z}$ and for all pairs $u,v\in\mathcal{Y}$ and such that $\widetilde{w}=v\overline{w}$ and $v=uw$ for some path $w$ in $A$. Note that, \cite[Proposition 2]{BM03} guaranties the uniqueness of $w$ when $u\neq v$; in the case $u=v$, we will consider the trivial path $w=e_{t(u)}$.

\vspace{0,5cm}

Now, let us consider a morphism $\varphi^{\bullet} \in \textup{Hom}_{\mathbf{C}^b(\textup{proj} A)}\big(\textbf{P}^\bullet, \widetilde{{\bf P}}^\bullet\big)$. To represent $\varphi^{\bullet}$ as a matrix, at each place $j \in \mathbb{Z}$, the morphism $\varphi^{\bullet}$ is a homomorphism $\varphi^{j}$ from the projective $\textbf{P}^j$ to the projective $\widetilde{{\bf P}}^j$, and as above,  $\varphi^{j}$ is a block matrix between direct sums of indecomposable projective. Thus, as we did with the differentials, if we denote by $\phi_{w,j}$ the blocks in $\varphi^j$, we can represent $\varphi^j$ by a formal sum
    \begin{equation}
    \label{eq:form-morph}
    \varphi^j:\sum_{w \in \mathbf{Pa}}p(w){\phi}_{w,j}.
	\end{equation}      
    
From commutative diagram 
    $$
    \xymatrix{
    \mathbf{P}^\bullet\ar[d]_{\varphi^\bullet}: &\cdots\ar[r]  &\mathbf{P}^j\ar[d]_{\varphi^j}\ar[r]^{\partial^j} & \mathbf{P}^{j+1}\ar[d]^{\varphi^{j+1}}\ar[r]    & \cdots\\
\widetilde{\mathbf{P}}^\bullet: &\cdots\ar[r]  &\widetilde{\mathbf{P}}^j\ar[r]_{\widetilde{\partial}^j} & \widetilde{\mathbf{P}}^{j+1}\ar[r]    & \cdots}
	$$
we obtain, under the
notations above,  the following equation:    

    \begin{equation}\label{Formula 1}
     \sum_{\substack{w=w_1w_2 \\ w_1, w_2 \in \mathbf{Pa} }}p(w)\phi_{w_1,j}\widetilde{\mathbf{A}}_{w_2,j}=\sum_{\substack{w'=w_3w_4 \\ w_3, w_4 \in \mathbf{Pa}}}p({w}')\mathbf{A}_{w_3,j}\phi_{w_4,j+1}.
 \end{equation}
 
Finally, the functor $F$ will be defined in morphisms $\varphi^j$ as a matrix in terms of its $\phi_{w,j}$. Summarizing, we have:


\begin{definition} 
Let $\mathbf{F}:\mathbf{C}^b(\textup{proj} A)\longrightarrow s(\mathscr{Y}(A),\Bbbk)$ be the functor defined as follows:

\vspace{,2cm}

In objects, $\mathbf{P}^{\bullet}\in\mathbf{C}^b(\textup{proj} A)$, each $([u,i],[v,j])$th block is given by
   $$ \mathbf{F}(\mathbf{P}^{\bullet} )^{[v,j]}_{[u,i]}= 
   \begin{cases}
        \mathbf{A}_{w,i},  & \text{if  $j=i+1$ , $v=uw$ , $w \in \mathbf{Pa},$}   \\
         0, & \text{ otherwise,}
     \end{cases}$$
where the block $\mathbf{F}(\mathbf{P}^\bullet)_{[u,i]}($respectively, $\mathbf{F}(\mathbf{P}^\bullet)^{[u,i]})$ has $d_{t(u),i}$ rows (respectively, columns) and $\partial^i:\sum_{w \in \mathbf{Pa}}p(w)\mathbf{A}_{w,i}$ is the differential $\mathbf{P}^i\longrightarrow \mathbf{P}^{i+1}$.

\vspace{,2cm}
 
In morphisms, $\varphi^{\bullet}$ in $\mathbf{C}^b(\textup{proj} A)$, each $([u,i],[v,j])$th block is given by
    $$
    \mathbf{F}(\varphi^{\bullet} )^{[v,j]}_{[u,i]}=
    \begin{cases}
        \phi_{w,i},  & \text{if  $j=i$ , $ v=uw $,  $w \in \mathbf{Pa}$},   \\
         0, & \text{ otherwise. } 
     \end{cases}
     $$
where $\varphi^i:\sum_{w \in \mathbf{Pa}}p(w){\phi}_{w,i}$.
\end{definition}


Note that, the condition that all products $\partial^j\partial^{j+1}$ are equal to zero is translated as the requirement that all products of consecutive blocks are equal to zero, or equivalently, that the matrix $\mathbf{F}(\mathbf{P}^{\bullet})$ has a square equal to zero. The equality \eqref{Formula 1} implies $\mathbf{F}(\mathbf{P}^\bullet)\mathbf{F}(\varphi^\bullet)=\mathbf{F}(\varphi^\bullet)\mathbf{F}({\widetilde{\mathbf{P}}}^\bullet)$. Furthermore, $\mathbf{F}$ preserves composition due to the conventions we have chosen regarding the action of the matrices $\phi_{w,i}$ on the domain of $\varphi_i$ and the convention of arrow composition in the quiver $Q$.


\begin{remark}
The functor $\mathbf{F}$ is a slight extension of the original functor defined in \cite[Section 3]{BM03}. In that work, the functor is defined over the full subcategory of complexes of projective modules, where the image of each differential map is contained within the radical of the corresponding projective module. The difference in our approach lies in the consideration of trivial paths. 
\end{remark}


Before proceeding with the results in this section and exploring the connection between the triangulated structures of $\mathbf{K}^b(\textup{proj} A)$ and $\mathcal{K}(\mathscr{Y}(A),\Bbbk)$, let us first consider the examples \ref{example 3.4} and \ref{ex:ima-cone}.


\begin{example}
\label{example 3.4}
Let us consider the gentle algebra $A_1=\Bbbk(Q,I)$ from Example \ref{ex:gentle}(i). Its poset and involution is given in Example \ref{ex:Y(A)}. To the complex 
   $$
    \xymatrix{
    \mathbf{P}^\bullet: & \cdots\ar[r] &0\ar[r]  & \mathbf{P}^1 =P_1\ar[r]^-{\partial^1}    &  \mathbf{P}^2=P_1^2 \oplus P_2\ar[r] & 0\ar[r]& \cdots,
    }
    $$
with 
	$$
    \partial^1=\left(\begin{array}{ccc}
    2p(x) & p(e_1) &  p(a)+3p(ay)+2p(xay)
    \end{array}\right),
    $$
we calculate the matrix $\mathbf{F}(\mathbf{P}^\bullet)$. Since, the differential maps  correspond to the formal sums 
    $$
    \partial^j:
    { p(e_1)\mathbf{A}_{e_1,j}}+ p(e_2)\mathbf{A}_{{e_2,j}}+p(x)\mathbf{A}_{x,j}+
    p(a)\mathbf{A}_{a,j}+
    p(y)\mathbf{A}_{y,j}+ p(xa)\mathbf{A}_{xa,j}+
    {p(ay)\mathbf{A}_{ay,j}} +
    {p(xay)\mathbf{A}_{xay,j}},
    $$
 then 
	$$
	\begin{array}{ccccc}
 	&  \partial^1:p(e_1)\left(\begin{matrix}
    0 & 1
\end{matrix}\right) +p(e_2)\emptyset +p(x)\left(\begin{matrix}
    2 & 0
\end{matrix}\right) + p(a)\left(\begin{matrix}
    1
\end{matrix}\right) + p(y)\emptyset+p(xa)\left(\begin{matrix}
   0
\end{matrix}\right)+ p(ay)\left(\begin{matrix}
    3
\end{matrix}\right)+ p(xay)\left(\begin{matrix}
    2
\end{matrix}\right),
\end{array}
	$$
where $\emptyset$ indicates that this block matrix is empty. Hence,
{\small{$$\begin{array}{ccclllll}
 \mathbf{F}(\mathbf{P}^\bullet)^{[e_1,2]}_{[e_1,1]}=\mathbf{F}(\mathbf{P}^\bullet)^{[x,2]}_{[x,1]}=\left(\begin{matrix}
     0 & 1
 \end{matrix}\right),  & \mathbf{F}(\mathbf{P}^\bullet)^{[x,2]}_{[e_1,1]}=
 \left(\begin{matrix}
     2 & 0
     \end{matrix}\right),
    & 
  \mathbf{F}(\mathbf{P}^\bullet)^{[xay,2]}_{[e_1,1]}=\left(\begin{matrix}
      2 
 \end{matrix}\right), \\\\ 
   \mathbf{F}(\mathbf{P}^\bullet)^{[xa,2]}_{[e_1,1]}=\left(\begin{matrix}
     0 
 \end{matrix}\right) & \mathbf{F}(\mathbf{P}^\bullet)^{[xa,2]}_{[x,1]}=\left(\begin{matrix}
     1 
 \end{matrix}\right), &  \mathbf{F}(\mathbf{P}^\bullet)^{[xay,2]}_{[x,1]}=\left(\begin{matrix}
     3 
 \end{matrix}\right).
\end{array}$$}}


The table below describes the size of  Bondarenko's matrix. Recall that the block $\mathbf{F}(\mathbf{P}^\bullet)_{[u,i]}$ (respectively, $\mathbf{F}(\mathbf{P}^\bullet)^{[u,i]}$) has $d_{t(u),i}$ rows (respectively, columns). Specifically, $d_{t(u),i}=0$ for all $i\not=1,2,3$ 
	$$
	\begin{array}{|c||c|c|c|c|}
	\hline
               [u,i]\in \mathscr{Y}(A) & {[e_1,i] }&{[x,i]}&{[xa,i]}&{[xay,i]}\\
             \hline\hline
d_{t(u),i}& d_{1,1}=1&d_{1,1}=1 & d_{2,1}=0 & d_{2,1}=0  \\

& d_{1,2}=2&d_{1,2}=2 & d_{2,2}=1 & d_{2,2}=1 \\
\hline
	\end{array}
	$$
Therefore, 
we obtain the following  object in $s(\mathscr{Y}(A),\Bbbk)$ 
	$$
\mathbf{F}(\mathbf{P}^\bullet) ={\footnotesize{\left(\begin{array}{c||c|c|cc|cc|c|c}
                 & {\blue [e_{1},1] }&\blue{[x,1]}&\blue{[e_{1},2]}&&\blue{[x,2]}&&\blue{[xa,2]}&\blue{[xay,2]}\\
             \hline\hline
{\blue[e_1,1]} & 0& 0&{\red \bf 0} & {\red \bf 1}&{\red \bf 2} & {\red \bf 0}& {\red \bf 0} & {\red \bf 2}\\
\hline
{\blue[x,1]}  & 0& 0&0 & 0&{\red \bf 0} & {\red \bf 1}& {\red \bf 1} &{\red \bf 3}\\
\hline
{\blue[e_1,2]} & 0& 0&0 & 0&0 & 0& 0 &0\\
& 0& 0&0 & 0&0 & 0& 0 &0\\
\hline
{\blue[x,2]} & 0& 0&0 & 0&0 & 0& 0 &0\\
& 0& 0&0 & 0&0 & 0& 0 &0\\
\hline
{\blue[xa,2]}  & 0& 0&0 & 0&0 & 0& 0 &0\\
\hline
{\blue[xay,2]}  & 0& 0&0 & 0&0& 0& 0 &0\\
\end{array}\right)}}
	$$
	
\vspace{,3cm}

	Now, let us determine the matrix $\mathbf{F}(\varphi^\bullet)$ for any $\varphi^\bullet \in \textup{End}_{\mathbf{C}^b(\textup{proj}A)}(\mathbf{P}^\bullet)$. The condition $\partial^1\varphi^2=\varphi^1\partial^1$ implies the existence of $\alpha,\lambda,\delta,\epsilon,\beta,\gamma \in \Bbbk$ such that
	$$
    \varphi^1
    =\left(\begin{array}{c}
    \lambda p(x) +  \beta p(e_1)
 \end{array}\right),
 \ \ \  
    \varphi^2=
 \left(\begin{array}{ccc}
    \alpha p(x)+  \beta p(e_1) & \gamma p(x) &  \delta p(xa)+\epsilon p(xay)\\
     0 & \lambda p(x)+\beta p(e_1) & \lambda p(xa) + 3\lambda p(xay)\\ 
    0& 0& \beta p(e_2) 
 \end{array}\right).
    $$ 
Since, each morphism $\varphi^j$ is  represented by the formal sums 
   \begin{align*}
    \varphi^j&:
    { p(e_1)\phi_{e_1,j}}+ p(x)\phi_{x,j}+
    p(a)\phi_{a,j}+
    p(y)\phi_{y,j}+ p(xa)\phi_{xa,j}+
    {p(ay)\phi_{ay,j}} +
    {p(xay)\phi_{xay,j}},
   \end{align*}
we have
\begin{align*}
  \varphi^1 &:
    { p(e_1)\left(\begin{matrix}
    \beta
\end{matrix}\right)}+ p(x)\left(\begin{matrix}
    \lambda
\end{matrix}\right)+
    p(a)\emptyset+
    p(y)\emptyset+ p(xa)\emptyset+
    {p(ay)\emptyset} +
    {p(xay)\emptyset},\\ \\
    \varphi^2 & :
    { p(e_1)\left(\begin{array}{ccc}
     \!\! \beta  & 0 \!\!\!\\
     \!\! 0 & \beta \!\!\!
 \end{array}\right)}+ p(e_2)\left(\begin{matrix}
 \beta
 \end{matrix}\right)+ p(x)\left(\begin{array}{ccc}
    \!\! \alpha & \gamma \!\!\!\\
     \!\! 0 & \lambda \!\!\!
 \end{array}\right)+
   p(a)\left(\begin{array}{ccc}
   \!\! 0 \!\!\!\\
   \!\! 0 \!\!\!
 \end{array}\right)+
    p(y)\left( 0 \right)+ p(xa)\left(\begin{array}{ccc}
    \!\! \delta \!\!\!\\
    \!\! \lambda \!\!\! 
 \end{array}\right)+
    {p(ay)}\left(\begin{array}{ccc}
    \!\! 0\!\!\!\\
    \!\! 0\!\!\! 
 \end{array}\right) +
    {p(xay)}\left(\begin{array}{ccc}
    \!\! \epsilon\!\!\!\\
     \!\! 3\lambda\!\!\! 
 \end{array}\right).
\end{align*}
Then, we have the following morphism in $s(\mathscr{Y}(A),\Bbbk)$
	$$
\mathbf{F}(\varphi^\bullet) ={\footnotesize{\left(\begin{array}{c||c|c|cc|cc|c|c}
                 & {\blue [e_{1},1] }&\blue{[x,1]}&\blue{[e_{1},2]}&&\blue{[x,2]}&&\blue{[xa,2]}&\blue{[xay,2]}\\
             \hline\hline
{\blue[e_1,1]} & {\red\pmb{\beta}}& {\red\pmb{\lambda}}&0 & 0&0 & 0& 0 &0 \\
\hline
{\blue[x,1]}  & 0& {\red\pmb{\beta}}&0 & 0&0 & 0& 0&0\\
\hline
{\blue[e_1,2]} & 0& 0&{\red\pmb{\beta}}& {\red\pmb{0}}&{\red\pmb{\alpha}}& {\red\pmb{\gamma}}& {\red\pmb{\delta}}&{\red\pmb{\epsilon}}\\
& 0& 0&{\red\pmb{0}}& {\red\pmb{\beta}}&{\red\pmb{0}}& {\red\pmb{\lambda}}& {\red\pmb{\lambda}}&{\red\pmb{3\lambda}}\\
\hline
{\blue[x,2]} & 0& 0&0 & 0&{\red\pmb{\beta}}& {\red\pmb{0}}& {\red\pmb{0}}&{\red\pmb{0}}\\
& 0& 0&0 & 0&{\red\pmb{0}}& {\red\pmb{\beta}}& {\red\pmb{0}}&{\red\pmb{0}}\\
\hline
{\blue[xa,2]}  & 0& 0&0 & 0&0 & 0& {\red\pmb{\beta}}&{\red\pmb{0}}\\
\hline
{\blue[xay,2]}  & 0& 0&0 & 0&0& 0& 0 &{\red\pmb{\beta}}\\
\end{array}\right)}}
$$

\end{example}


\vspace{,3cm}

It is well known that $\mathbf{K}^b(\textup{proj} A)$ has triangulated structure (see \cite[Theorem 2.3.1, p. 11]{KZ98}). The shift functor is denoted by $[1]$ for $\mathbf{C}^b(\textup{proj} A)$ and $\mathbf{K}^b(\textup{proj} A)$, respectively. For any morphism $\varphi^\bullet$, the distinguished triangles are (up to isomorphism) of the form $\xymatrix{\mathbf{P}^\bullet\ar[r]^{\varphi^\bullet}	& \widetilde{\mathbf{P}}^\bullet\ar[r]^-{\iota_{\widetilde{P}}^\bullet}	& \mathbf{C}(\varphi)^\bullet\ar[r]^{\pi_{P[1]}^\bullet}	 	&   \mathbf{P}^{\bullet}[1]}$, where $\mathbf{C}(\varphi)^\bullet$ is called the \emph{mapping cone} of $\varphi^\bullet$. This triangle is called \emph{standard triangle} associated with $\varphi^\bullet$ in $\mathbf{K}^b(\textup{proj} A)$. Recall that $\mathbf{C}(\varphi)^\bullet$ is the complex in $\mathbf{C}^b(\textup{proj}A)$ with differential $\partial_{\varphi}^j$ given by  
	$$
	\xymatrix{ \mathbf{P}^{j+1}\oplus \widetilde{\mathbf{P}}^{j}\ar[rrrr]^-{\partial_{\varphi}^j:=\left(\begin{matrix}
     -{\partial}^{j+1}  & \varphi^{j+1} \\
     0  & \widetilde{\partial}^j
 \end{matrix}\right)}&&&&  \mathbf{P}^{j+2}\oplus \widetilde{\mathbf{P}}^{j+1} }.
 	$$
 
Note that, from representations \eqref{form:diff} and \eqref{eq:form-morph}, each differential $\partial_{\varphi}^{j}$ is represented by the formal sum
\begin{align}\label{form:cone}
\partial^j_{\varphi}=\left(\begin{matrix}
     -{\partial}^{j+1}  & \varphi^{j+1} \\
     0  & \widetilde{\partial}^j
 \end{matrix}\right):\left(\begin{matrix}
     -\sum\limits_{w \in \mathbf{Pa}}p(w)\mathbf{A}_{w,j+1}  & \sum\limits_{w \in \mathbf{Pa}}p(w)\phi_{w,j+1} \\
     0  & \sum\limits_{w \in \mathbf{Pa}}p(w)\widetilde{\mathbf{A}}_{w,j}
 \end{matrix}\right).
\end{align}

 
\begin{remark}
\label{diagramkappa}
It is straightforward to verify that $\mathbf{F}$ is an additive functor and that the equality $\llbracket -\rrbracket \circ \mathbf{F} = \mathbf{F} \circ [1]$ holds. 
\end{remark}
 
 
\begin{example}
\label{ex:ima-cone}
The mapping cone of endomorphism $\varphi^\bullet$ from Example \ref{example 3.4} is 
 $$
    \xymatrix{
   \mathbf{C}(\varphi^\bullet):  \cdots\ar[r] &0\ar[r]  & \mathbf{P}^0=P_1\ar[r]^-{\partial_{\varphi}^0} &   \mathbf{P}^1=P_1^2 \oplus P_2\oplus P_1\ar[r]^-{{\partial_{\varphi}^1}} & \mathbf{P}^2 =P_1^2 \oplus P_2\ar[r]&  0\ar[r]& \cdots,
    }
    $$
    with differentials
\begin{align*}
\partial_{\varphi}^0 & =\left(\begin{array}{cccc}
-2p(x) & -p(e_1) &  -p(a)-3p(ay)-2p(xay) & \lambda p(x) + \beta p(e_1)
\end{array}\right)\\
\\
\partial_{\varphi}^1&= \left(\begin{array}{cccc}
 \alpha p(x)+  \beta p(e_1)  & \gamma p(x) &  \delta p(xa)+\epsilon p(xay)\\
     0 & \lambda p(x)+\beta p(e_1) &\lambda p(xa)+ 3\lambda p(xay)\\ 
    0& 0& \beta p(e_2) \\
     2p(x) & p(e_1) &  p(a)+3p(ay)+2p(xay)
\end{array}\right)
\end{align*}
The representations of the differential maps are given by 
\begin{align*}
\partial^0_{\varphi}&:p(e_1)\!\left(\begin{matrix}
    0 \  -1 \  \beta
\end{matrix}\right)\!+p(e_2)\emptyset \!+p(x)\!\left(\begin{matrix}
    -2 \   0\  \lambda
\end{matrix}\right)\! - \!p(a)\!\left(\begin{matrix}
    1
\end{matrix}\right)\! + \! p(y)\emptyset\!+\!p(xa)\!\left(\begin{matrix}
   0
\end{matrix}\right)\! - \! p(ay)\!\left(\begin{matrix}
    3
\end{matrix}\right)\! - \! p(xay)\!\left(\begin{matrix}
    2
\end{matrix}\right)\\
\\
 \partial_{\varphi}^1 &:
    { p(e_1)\left(\begin{array}{ccc}
     \!\! \beta  & 0 \!\!\!\\
     \!\! 0 & \beta \!\!\! \\
     \!\! 0 & 1 \!\!\!
 \end{array}\right)}
 + p(e_2)\left(\beta\right)
 +p(x)\left(\begin{array}{ccc}
   \!\! \alpha & \gamma \!\!\!\\
   \!\!  0 & \lambda \!\!\!\\
   \!\!  2 & 0 \!\!\! 
 \end{array}\right)
 +p(a)\left(\begin{array}{ccc}
   \!\! 0 \!\!\! \\
   \!\! 0 \!\!\!\\
   \!\! 1 \!\!\!
 \end{array}\right)
 +p(y)\left( 0 \right)
 +p(xa)\left(\begin{array}{ccc}
   \!\! \delta \!\!\!\\
   \!\!  \lambda\!\!\! \\
   \!\!  0 \!\!\!
 \end{array}\right)
 +{p(ay)}\left(\begin{array}{ccc}
   \!\! 0 \!\!\!\\
   \!\! 0 \!\!\!\\
   \!\! 3 \!\!\!
 \end{array}\right) 
 +{p(xay)}\left(\begin{array}{ccc}
   \!\! \epsilon\!\!\! \\
   \!\! 3\lambda\!\!\! \\
   \!\! 2\!\!\!
 \end{array}\right).
\end{align*}
Thus, using a similar process to the one we used to calculate the block matrix $\mathbf{F}(\mathbf{P}^\bullet)$, we have for $\mathbf{F}(\mathbf{C}(\varphi)^\bullet)$ the following table  
$$
\begin{array}{|c||c|c|c|c|}
\hline
               [u,i]\in \mathscr{Y}(A) & {[e_1,i] }&{[x,i]}&{[xa,i]}&{[xay,i]}\\
             \hline\hline
& d_{1,0}=1&d_{1,0}=1 & d_{2,0}=0 & d_{2,0}=0  \\
d_{t(u),i} 
& d_{1,1}=3&d_{1,1}=3 & d_{2,1}=1 & d_{2,1}=1 \\
& d_{1,2}=2&d_{1,2}=2 & d_{2,2}=1 & d_{2,2}=1 \\
\hline
\end{array}
$$
Therefore, we obtain the following  object  $\mathbf{F}(\mathbf{C}(\varphi)^\bullet)$ in $s(\mathscr{Y}(A),\Bbbk)$, which is given by 
{\footnotesize {$$
\left(\begin{array}{c||c|c|c|c|c|c|c|c|c|c}
                 & {\blue [e_{1},0] }&\blue{[x,0]}&\blue{[e_{1},1]}&\blue{[x,1]}&\blue{[xa,1]}&\blue{[xay,1]}& \blue{[e_{1},2]}&\blue{[x,2]}&\blue{[xa,2]}&\blue{[xay,2]}\\
             \hline\hline
{\blue[e_1,0]} & 0& 0&{\red \bf 0} \ \ {\red \bf -1} \ \ {\red\pmb{\beta}} & {\red \bf -2}\ \  {\red \bf 0} \ \  {\red\pmb{\lambda}} & {\red \bf 0} & {\red \bf -2}& 0\ \  0& 0\ \  0& 0& 0 \\
\hline
{\blue[x,0]}  & 0& 0&0 \ \ 0\ \ 0 & {\red \bf 0}\ \  {\red \bf -1} \ \ {\red\pmb{\beta}} & {\red \bf -1} & {\red \bf -3}& 0\ \  0& 0\ \  0& 0& 0 \\
\hline
 & 0& 0&0 \ \  0\ \ 0 & 0\ \  0 \ \  0 & 0 & 0& {\red\pmb{\beta}} \ \  {\red \bf 0}& {\red\pmb{\alpha}}\ \  {\red\pmb{\gamma}}& {\red\pmb{\delta}}& {\red\pmb{\epsilon}}\\
{\blue[e_1,1]}& 0& 0&0 \ \  0\ \ 0 & 0\ \  0 \ \  0 & 0 & 0& {\red \bf 0} \ \  {\red\pmb{\beta}}& {\red \bf 0} \ \ {\red\pmb{\lambda}}& {\red\pmb{\lambda}}& {\red\pmb{3\lambda}} \\
& 0& 0&0 \ \  0 \ \ 0 & 0\ \  0 \ \  0 & 0 & 0& {\red \bf 0} \ \ {\red \bf 1}& {\red \bf 2}\ \  {\red \bf 0}& {\red \bf 0}& {\red\pmb{ 2 }} \\ 
\hline
 & 0& 0&0 \ \ 0\ \ 0 & 0\ \  0 \ \ 0 & 0 & 0& 0\ \ 0& {\red\pmb{ \beta}}\ \  {\red \bf 0}& {\red \bf 0}& {\red \bf 0} \\
{\blue[x,1]}& 0& 0&0 \ \  0 \ \ 0 & 0\ \  0 \ \  0 & 0 & 0& 0\ \  0& {\red \bf 0}\ \  {\red\pmb{ \beta}}& {\red \bf 0}& {\red \bf 0} \\
& 0& 0&0 \ \  0\ \ 0 & 0\ \  0 \ \  0 & 0 & 0& 0\ \  0& {\red \bf 0}\ \  {\red \bf 1}& {\red \bf 1}& {\red \bf 3} \\
\hline
{\blue[xa,1]}  & 0& 0&0 \ \  0 \ \ 0 & 0\ \  0 \ \ 0 & 0 & 0& 0 \ \ 0& 0\ \  0& {\red\pmb{ \beta}}& {\red \bf 0} \\
\hline
{\blue[xay,1]}  & 0& 0&0 \ \  0\ \ 0 & 0\ \  0 \ \  0 & 0 & 0& 0\ \  0& 0\ \  0& 0& {\red\pmb{ \beta}} \\
\hline
{\blue[e_1,2]} & 0& 0&0 \ \  0\ \ 0 & 0\ \  0 \ \  0 & 0 & 0& 0\ \  0& 0 \ \ 0& 0& 0 \\
& 0& 0&0 \ \  0\ \ 0 & 0\ \  0 \ \  0 & 0 & 0& 0\ \  0& 0 \ \ 0& 0& 0 \\
\hline
{\blue[x,2]}  & 0& 0&0 \ \  0\ \ 0 & 0\ \  0 \ \  0 & 0 & 0& 0\ \  0& 0 \ \ 0& 0& 0 \\
& 0& 0&0 \ \  0\ \ 0 & 0\ \  0 \ \  0 & 0 & 0& 0\ \  0& 0 \ \ 0& 0& 0 \\
\hline
{\blue[xa,2]} & 0& 0&0 \ \  0\ \ 0 & 0\ \  0 \ \  0 & 0 & 0& 0\ \  0& 0 \ \ 0& 0& 0 \\
\hline
{\blue[xay,2]} & 0& 0&0 \ \  0\ \ 0 & 0\ \  0 \ \  0 & 0 & 0& 0\ \  0& 0 \ \ 0& 0& 0 \\
\end{array}\right)
$$}}

\end{example}



\vspace{,3cm}

The following proposition establishes a connection between standard triangles in $\mathbf{C}^b(\textup{proj}A)$ and $\mathcal{K}$-standard triangles in $\mathcal{K}(\mathscr{Y}(A),\Bbbk)$.


\begin{proposition}
\label{standard}
The functor $\mathbf{F}:\mathbf{C}^b(\textup{proj}A)\longrightarrow s(\mathscr{Y}(A),\Bbbk)$ sends standard triangles of $\mathbf{C}^b(\textup{proj}A)$ to $\mathcal{K}$-standard triangles of $s(\mathscr{Y}(A),\Bbbk)$.

\end{proposition}


\begin{proof}
Let $\varphi^\bullet$ be a morphism from $\mathbf{P}^\bullet$ to ${\widetilde{\mathbf{P}}}^\bullet$ in $\mathbf{C}^b(\textup{proj}A)$, as the differentials of mapping cone $\mathbf{C}({\varphi})^\bullet$ are represented by \eqref{form:cone}, applying the functor $\mathbf{F}$ on $\mathbf{C}(\varphi)^\bullet$, we have that the $([u,i],[v,j])$th block is 
\begin{align*}
 \mathbf{F}(\mathbf{C}({\varphi})^\bullet )^{[v,j]}_{[u,i]}&= 
   \begin{cases}
         \left( \begin{array}{cc}
  -\mathbf{A}_{w,i+1}  & \phi_{w,i+1}\\ 
   0  & \widetilde{\mathbf{A}}_{w,i}
\end{array}\right) &,\ \text{if  $j=i+1$ , $v=uw$ , $w \in \mathbf{Pa},$}  \\\\
         \ \ \ \ \ \ \ \ \ \ \ \ \ \ \ 0 &,\ \text{otherwise.}
     \end{cases}\\
     &=  \left( \begin{array}{cc}
  -\mathbf{F}(\mathbf{P}^\bullet)^{[v,j+1]}_{[u,i+1]}  & \mathbf{F}(\varphi^\bullet)^{[v,j]}_{[u,i+1]}\\ 
   \mathbf{0}  & \mathbf{F}(\widetilde{\mathbf{P}}^\bullet)^{[v,j]}_{[u,i]}
\end{array}\right). &  
\end{align*} 
 Consequently, it follows that $\mathbf{F}(\mathbf{C}({\varphi})^\bullet)={\Omega_{\mathbf{F}(\varphi^\bullet)}}$. We can conclude similarly that  $\mathbf{F}(\iota_{\widetilde{\mathbf{P}}}^\bullet)=\iota_{\mathbf{F}({\widetilde{\mathbf{P}}}^\bullet)}$ and $\mathbf{F}(\pi_{{\mathbf{P}[1]}}^\bullet )=\pi_{\llbracket\mathbf{F}({{\mathbf{P}}}^\bullet)\rrbracket}$. Finally, applying the functor ${\mathbf{F}} $ on  the  standard triangle in $\mathbf{C}^b(\textup{proj}A)$  
 $$
    \xymatrix{ \mathbf{P}^\bullet\ar[r]^{\varphi^{\bullet}}  &\widetilde{\mathbf{P}}^\bullet\ar[r]^-{\iota_{\widetilde{\mathbf{P}}}^\bullet}    & \mathbf{C}(\varphi)^\bullet\ar[r]^-{\pi_{{\mathbf{P}}[1]}^\bullet} & {\mathbf{P}}^\bullet[1]},
    $$ 
from Remark \ref{diagramkappa} we obtain the triangle in $s(\mathcal{Y}(A),\Bbbk)$ 
$$  \xymatrix{ {\mathbf{F}}(\mathbf{P}^\bullet)\ar[r]^{{\mathbf{F}}(\varphi^{\bullet})}  &{\mathbf{F}} (\widetilde{\mathbf{P}}^\bullet)\ar[r]^-{\iota_{{\mathbf{F}}(\widetilde{\mathbf{P}}^\bullet)}}    & {\Omega_{{\mathbf{F}} (\varphi^\bullet)}}\ar[r]^-{ \pi_{\llbracket {\mathbf{F}}({\mathbf{P}}^\bullet)\rrbracket} } & {\llbracket{\mathbf{F}} (\mathbf{P}}^\bullet)\rrbracket},$$
which is  $\mathcal{K}$-standard triangle. 

\end{proof}
\begin{example}
Consider the gentle algebra $A_2=\Bbbk Q$ from Example \ref{ex:gentle}(ii). The poset and involution for this algebra are provided in Example \ref{ex:Y(A)}, which align with those in Example \ref{ex:triangulo}, except that $u$ is replaced by $e_{s(a)}$ and $v$ is replaced by $e_{s(b)}$. Consider the endomorphism $\varphi^\bullet =(\varphi^1 \ \ \varphi^2)\in \textup{Hom}_{\mathbf{C}^b(\textup{proj}A)}(\mathbf{P}^\bullet , \mathbf{P}^\bullet [1])$, with 
	$$
    \xymatrix{
    \mathbf{P}^\bullet: & \cdots\ar[r] &0\ar[r]  & \mathbf{P}^1 =P_1\ar[r]^-{\partial^1}    &  \mathbf{P}^2=P_2\ar[r]& 0\ar[r]& \cdots
    }
    $$
$\partial^1=p(a)-p(b)$,  $ \varphi^1= p(a)+p(b) $ and $ \varphi^2= 0 $. 
The $\mathcal{K}$-standard triangle in $s(\mathscr{Y},\Bbbk)$ of the Example \ref{ex:triangulo} is the image over $\mathbf{F}$ of the standard triangle
 	$$ 
 	\xymatrix{\mathbf{P}^\bullet\ar[r]^{\varphi^{\bullet}}  &{\mathbf{P}}^\bullet[1]\ar[r]^-{\iota_{\widetilde{\mathbf{P}}}^\bullet}    & \mathbf{C}(\varphi)^\bullet\ar[r]^-{\pi_{{\mathbf{P}}[1]}^\bullet} & {\mathbf{P}}^\bullet[1]}.
    $$
\end{example}


\vspace{,3cm}

Recall that two morphisms $\varphi^\bullet, \psi^\bullet:\mathbf{P}^\bullet \longrightarrow \widetilde{\mathbf{P}}^\bullet$ are said to be \emph{homotopic}, denoted by $\varphi^\bullet \sim \psi^\bullet$, if there exists a sequence of morphisms $s^j: \mathbf{P}^j \longrightarrow \widetilde{\mathbf{P}}^{j-1}$ such that $\varphi^j - \psi^j = s^j \widetilde{\partial}^{j-1} + \partial^{j} s^{j+1}$ for all $j \in \mathbb{Z}$. 

The following example supports Remark \ref{remark2.1}(iii) by providing two homotopic morphisms $\varphi^{\bullet}\sim \psi^{\bullet}$ such that $\mathbf{F}(\varphi^{\bullet})\not\equiv \mathbf{F}(\psi^{\bullet})$.


\begin{example}
\label{exem3.6}
Let us consider the gentle algebra $A_1=\Bbbk(Q,I)$ from Example \ref{ex:gentle}(i). Its poset and involution is given in Example \ref{ex:Y(A)}. Take the endomorphism $\varphi^\bullet =(\varphi^1 \ \ \varphi^2)\in \textup{End}_{\mathbf{C}^b(\textup{proj}A)}(\mathbf{P}^\bullet)$, with 
	$$
    \xymatrix{
    \mathbf{P}^\bullet: & \cdots\ar[r] &0\ar[r]  & \mathbf{P}^1 =P_1\ar[r]^-{\partial^1}    &  \mathbf{P}^2=P_1\ar[r]& 0\ar[r]& \cdots
    }
    $$
and $\varphi^1=\varphi^2=\partial^1=p(x)+p(e_1)$. It  is easy see that $\varphi^\bullet \sim \mathbf{0}^\bullet$. Applying the functor $\mathbf{F}$ we obtain (similarly to Example \ref{example 3.4}):

{\small{\begin{align*}
\mathbf{F}(\mathbf{P}^\bullet) &=\left(\begin{array}{c||c|c|c|c}
                 & {\blue [e_{1},1] }&\blue{[x,1]}&\blue{[e_{1},2]}&\blue{[x,2]}\\
             \hline\hline
{\blue[e_1,1]}  & 0& 0& \textcolor{red}{1} &  \textcolor{red}{1}\\
\hline
{\blue[x,1]}   & 0& 0&0 &  \textcolor{red}{1}\\
\hline
{\blue[e_1,2]}  & 0& 0&0 & 0\\
\hline
{\blue[x,2]} & 0& 0&0 & 0\\
\end{array}\right) \ ,\ &\mathbf{F}(\varphi^\bullet) =\left(\begin{array}{c||c|c|c|c}
                 & {\blue [e_{1},1] }&\blue{[x,1]}&\blue{[e_{1},2]}&\blue{[x,2]}\\
             \hline\hline
{\blue[e_1,1]}  & \textcolor{red}{1}&  \textcolor{red}{1}&0 & 0\\
\hline
{\blue[x,1]}   & 0&  \textcolor{red}{1}&0 & 0\\
\hline
{\blue[e_1,2]}  & 0& 0& \textcolor{red}{1} &  \textcolor{red}{1}\\
\hline
{\blue[x,2]} & 0& 0&0 &  \textcolor{red}{1}\\
\end{array}\right)
\end{align*}}}

Naturally, any $\kappa$-matrix has to be the form
\begin{align*}
K&=\left(\begin{array}{c||c|c|c|c}
                 & {\blue [e_{1},1] }&\blue{[x,1]}&\blue{[e_{1},2]}&\blue{[x,2]}\\
             \hline\hline
{\blue[e_1,1]}  & a_{11}&  a_{12}&a_{13} & a_{14}\\
\hline
{\blue[x,1]}   & 0&  a_{22}&a_{23}& a_{24}\\
\hline
{\blue[e_1,2]}  & 0& 0& a_{33} &  a_{34}\\
\hline
{\blue[x,2]} & 0& 0&0 &  a_{44}\\
\end{array}\right)
 \end{align*}
with $a_{ij}\in\Bbbk$ for all $1\leq i\leq j \leq 4$. Note that $\mathbf{F}(\varphi^\bullet)\not=K\mathbf{F}(\mathbf{P}^\bullet)+\mathbf{F}(\mathbf{P}^\bullet)K$, which implies that $\mathbf{F}(\varphi^{\bullet})\not\equiv \mathbf{F}(\mathbf{0}^{\bullet})$. Therefore, the functor does not preserve the homotopy relation.

\end{example}


Our goal now is to establish a connection between the homotopy relation ``$\sim$'' and the relation ``$\simeq$'' (from Section~\ref{sec:2}) through the functor $\mathbf{F}$. To this end, we present the following technical lemma.

 
\begin{lemma}
\label{teclemma}
Let $\varphi^{\bullet} \in \textup{Hom}_{\mathbf{C}^b(\textup{proj} A)}\big(\mathbf{P}^\bullet, \widetilde{{\bf P}}^\bullet\big)$ be a morphism. If $\mathbf{F}(\varphi^\bullet) \simeq \mathbf{F}(\mathbf{0}^\bullet)=\mathbf{0}$  with $\mathcal{K}$-matrix $\mathcal{S}$, then 
	\begin{align*}    
   	\mathbf{F}(\varphi^{\bullet} )^{[uw,i]}_{[u,i]}
   	   &=\sum_{\substack{w_1w_2=w \\ w_1,  w_2 \in \textbf{Pa}}}\mathcal{S}^{[uw_1,i-1]}_{[u,i]}\mathbf{F}({\widetilde{\mathbf{P}}}^\bullet)^{[uw_1w_2,i]}_{[uw_1,i-1]}+ \sum_{\substack{w_3w_4=w \\ w_3,  w_4 \in \textbf{Pa} }}\mathbf{F}({{\mathbf{P}}}^\bullet)^{[uw_3,i+1]}_{[u,i]}\mathcal{S}^{[uw_3w_4,i]}_{[uw_3,i+1]}.
    \end{align*}
\end{lemma}


\begin{proof}
Since $\mathbf{F}(\varphi^\bullet)=\mathcal{S}\mathbf{F}({\widetilde{\mathbf{P}}}^\bullet)+ \mathbf{F}({{\mathbf{P}}}^\bullet)\mathcal{S}$, in the $([u,i],[uw,i])$th position we have 
\begin{align*}
\mathbf{F}(\varphi^{\bullet} )^{[uw,i]}_{[u,i]}&=\sum_{[r,k]\in \mathscr{Y}(A)}\mathcal{S}^{[r,k]}_{[u,i]}\mathbf{F}({\widetilde{\mathbf{P}}}^\bullet)^{[uw,i]}_{[r,k]}+ \sum_{[r,k]\in \mathscr{Y}(A)}\mathbf{F}({{\mathbf{P}}}^\bullet)^{[r,k]}_{[u,i]}\mathcal{S}^{[uw,i]}_{[r,k]}\\
 & = \sum_{r\in \mathbf{Pa}}\mathcal{S}^{[r,i-1]}_{[u,i]}\mathbf{F}({\widetilde{\mathbf{P}}}^\bullet)^{[uw,i]}_{[r,i-1]}+ \sum_{r\in \mathbf{Pa}}\mathbf{F}({{\mathbf{P}}}^\bullet)^{[r,i+1]}_{[u,i]}\mathcal{S}^{[uw,i]}_{[r,i+1]}\\
 & = \sum_{\substack{r\in \mathbf{Pa}\\ u\leq r}}\mathcal{S}^{[r,i-1]}_{[u,i]}\mathbf{F}({\widetilde{\mathbf{P}}}^\bullet)^{[uw,i]}_{[r,i-1]}+ \sum_{\substack{r\in \mathbf{Pa}\\ u\leq r}}\mathbf{F}({{\mathbf{P}}}^\bullet)^{[r,i+1]}_{[u,i]}\mathcal{S}^{[uw,i]}_{[r,i+1]}\\
 & = \sum_{\substack{r\in \mathbf{Pa}\\ u\leq r\leq uw}}\mathcal{S}^{[r,i-1]}_{[u,i]}\mathbf{F}({\widetilde{\mathbf{P}}}^\bullet)^{[uw,i]}_{[r,i-1]}+ \sum_{\substack{r\in \mathbf{Pa}\\ u\leq r\leq uw}}\mathbf{F}({{\mathbf{P}}}^\bullet)^{[r,i+1]}_{[u,i]}\mathcal{S}^{[uw,i]}_{[r,i+1]}
\end{align*}
All the equalities follow from item \textup{(iii)} of $\mathcal{K}$-matrix and the definition of the functor $\mathbf{F}$. For instance, in the second equality, we have $\mathbf{F}({\widetilde{\mathbf{P}}}^\bullet)^{[uw,i]}_{[r,k]} = 0$ for $k \neq i-1$ (respectively, $\mathbf{F}({{\mathbf{P}}}^\bullet)^{[r,k]}_{[u,i]} = 0$ for $k \neq i+1$). In the third equality, $\mathcal{S}^{[r,i-1]}_{[u,i]} = 0$ and $\mathbf{F}(\mathbf{P}^\bullet)^{[r,i+1]}_{[u,i]} = 0$ for $u > r$. Lastly, $\mathcal{S}^{[uw,i]}_{[r,i+1]} = 0$ and $\mathbf{F}({\widetilde{\mathbf{P}}}^\bullet)^{[uw,i]}_{[r,i-1]} = 0$ for $r > uw$.

 Now, for each $r\in \mathbf{Pa}$ such that $u \leq r \leq uw$, there exist $w_1, w_2 \in \mathbf{Pa}$ such that $r = uw_1$ and $uw = rw_2$, hence $w = w_1w_2$ by condition (ii) from definition of gentle algebra. Thus, we have 

 	\begin{align*}    
   	\mathbf{F}(\varphi^{\bullet} )^{[uw,i]}_{[u,i]}
   	   &=\sum_{\substack{w_1w_2=w \\ w_1,  w_2 \in \textbf{Pa}}}\mathcal{S}^{[uw_1,i-1]}_{[u,i]}\mathbf{F}({\widetilde{\mathbf{P}}}^\bullet)^{[uw_1w_2,i]}_{[uw_1,i-1]}+ \sum_{\substack{w_3w_4=w \\ w_3,  w_4 \in \textbf{Pa} }}\mathbf{F}({{\mathbf{P}}}^\bullet)^{[uw_3,i+1]}_{[u,i]}\mathcal{S}^{[uw_3w_4,i]}_{[uw_3,i+1]}
\end{align*}

\end{proof}


\begin{proposition}
\label{equiv}

Let $\varphi^{\bullet} \in \textup{Hom}_{\mathbf{C}^b(\textup{proj} A)}\big(\mathbf{P}^\bullet, \widetilde{{\bf P}}^\bullet\big)$ be a morphism. Then $\varphi^\bullet \sim \mathbf{0}^\bullet$ if and only if  $\mathbf{F}(\varphi^\bullet) \simeq \mathbf{F}(\mathbf{0}^\bullet)=\mathbf{0}$.
\end{proposition}


\begin{proof}
Suppose $\varphi^\bullet \sim \mathbf{0}^\bullet$. Then there exists a sequence of morphisms  $s^j:\mathbf{P}^j\longrightarrow \widetilde{\mathbf{P}}^{j-1}$ such that $\varphi^j=s^j\widetilde{\partial}^{j-1}+\partial^{j}s^{j+1}$.    From representations \eqref{form:diff} and \eqref{eq:form-morph}, we have
    \begin{align*}    
    \sum\limits_{w\in \textbf{Pa}}p(w)\phi_{w,j}&=\sum_{\substack{w_1w_2=w \\ w_1,  w_2 \in \textbf{Pa} }}p(w)\mathbf{S}_{w_1,j}\widetilde{\mathbf{A}}_{w_2,j-1}+\sum_{\substack{w_3w_4=w \\ w_3,  w_4 \in \textbf{Pa}}}p(w)\mathbf{A}_{w_3,j}\mathbf{S}_{w_4,j+1},
    \end{align*}
where each $s^j$  is represented  by the formal sum $s^j:\sum\limits_{w\in \textbf{Pa}}p(w)\mathbf{S}_{w,j}$. Applying the functor $\mathbf{F}$ on $\varphi^\bullet$

	\begin{align*}
	\mathbf{F}(\varphi^{\bullet} )^{[v,j]}_{[u,i]}&=\begin{cases}
      \phi_{w,j}  &,\ \ \text{if $j=i$, $v=uw$ and $w \in \mathbf{Pa}$},   \\
         0 &,\ \  \text{otherwise }. 
     \end{cases}\\
     &=\begin{cases}
      \sum\limits_{\substack{w_1w_2=w \\ w_1,  w_2 \in \textbf{Pa} }}\mathbf{S}_{w_1,j}\widetilde{\mathbf{A}}_{w_2,j-1}+\sum\limits_{\substack{w_3w_4=w \\ w_3,  w_4 \in \textbf{Pa}}}\mathbf{A}_{w_3,j}\mathbf{S}_{w_4,j+1}  &,\ \  \text{if $j=i$, $v=uw$ and $w \in \mathbf{Pa}$},   \\
         0 &,\ \  \text{otherwise }. 
     \end{cases}
	\end{align*}
Now, we consider the matrix $\mathcal{S}=\left(\mathcal{S}^{[v,j]}_{[u,i]}\right)_{{[u,i]},{[v,j]}\in \mathscr{Y}(A)}$, where the $([u,i],[v,j])$th block is defined by 
	$$
	\mathcal{S}^{[v,j]}_{[u,i]}=
	\begin{cases}
    \mathbf{S}_{w,i} &,\ \  \text{if $i=j+1$, $v=uw$ and $w \in \mathbf{Pa}$},   \\
    0 &,\ \  \text{otherwise. }
         \end{cases}
	$$
and each block $\mathcal{S}_{[u,i]}$ (respectively, $\mathcal{S}^{[u,i]}$) consists of $d_{t(u),i}$ rows (respectively, columns). 

It is straightforward to verify that, by construction, $\mathcal{S}$ is a $\mathcal{K}$-matrix, which consequently implies that $\mathbf{F}(\varphi^\bullet) \simeq \mathbf{F}(\mathbf{0}^\bullet)$.

Conversely, assume $\mathbf{F}(\varphi^\bullet) \simeq \mathbf{F}(\mathbf{0}^\bullet)=\mathbf{0}$. There exists a $\mathcal{K}$-matrix $\mathcal{S}$, such that $\mathbf{F}(\varphi^\bullet)=\mathcal{S}\mathbf{F}({\widetilde{\mathbf{P}}}^\bullet)+ \mathbf{F}({{\mathbf{P}}}^\bullet)\mathcal{S}$. Hence, the $([u,i],[v,j])$th block is 

	\begin{align*}
	\mathbf{F}(\varphi^{\bullet} )^{[v,j]}_{[u,i]}&=\sum_{[r,k]\in \mathscr{Y}(A)}\mathcal{S}^{[r,k]}_{[u,i]}\mathbf{F}({\widetilde{\mathbf{P}}}^\bullet)^{[v,j]}_{[r,k]}+ \sum_{[r,k]\in \mathscr{Y}(A)}\mathbf{F}({{\mathbf{P}}}^\bullet)^{[r,k]}_{[u,i]}\mathcal{S}^{[v,j]}_{[r,k]}.
	\end{align*}
From representations \eqref{form:diff}, \eqref{eq:form-morph} and Lemma~\ref{teclemma} we obtain 

	\begin{align*}    
   	\phi_{w,i}=\mathbf{F}(\varphi^{\bullet})^{[uw,j]}_{[u,i]}&=\sum_{\substack{w_1w_2=w \\ w_1,  w_2 \in \textbf{Pa} }}\mathcal{S}^{[uw_1,i-1]}_{[u,i]}\widetilde{\mathbf{A}}_{w_2,i-1}+\sum_{\substack{w_3w_4=w \\ w_3,  w_4 \in \textbf{Pa}}}\mathbf{A}_{w_3,i}\mathcal{S}^{[uw_3w_4,i]}_{[uw_3,i+1]}\\
   	&=\sum_{\substack{w_1w_2=w \\ w_1,  w_2 \in \textbf{Pa} }}\mathbf{S}_{w_1,i}\widetilde{\mathbf{A}}_{w_2,i-1}+\sum_{\substack{w_3w_4=w \\ w_3,  w_4 \in \textbf{Pa}}}\mathbf{A}_{w_3,j}\mathbf{S}_{w_4,i+1},
    \end{align*} 
where $
\mathbf{S}_{w,i+1}:=\mathcal{S}^{[uw,i]}_{[u,i+1]}$. Hence,
    \begin{align*}    
    \sum\limits_{w\in \textbf{Pa}}p(w)\phi_{w,i}&=\sum_{\substack{w_1w_2=w \\ w_1,  w_2 \in \textbf{Pa} }}p(w)\mathbf{S}_{w_1,i}\widetilde{\mathbf{A}}_{w_2,i-1}+\sum_{\substack{w_3w_4=w \\ w_3,  w_4 \in \textbf{Pa}}}p(w)\mathbf{A}_{w_3,i}\mathbf{S}_{w_4,i+1}.
    \end{align*}
Consequetenly, there exists a sequence of morphisms  $s^i:
    \mathbf{P}^i\longrightarrow \widetilde{\mathbf{P}}^{i-1},$ where each $s^i$  is represented  by the formal sum
$s^i:\sum\limits_{w\in \textbf{Pa}}p(w)\mathbf{S}_{w,i}$. Therefore, we have    
    $\varphi^j=s^j\widetilde{\partial}^{j-1}+\partial^{j}s^{j+1}$. 
    
\end{proof}


Let us mention two important consequences of the Propostion \ref{equiv}. First, it allows us to define a functor $\widehat{\mathbf{F}}:\mathbf{K}^b(\textup{proj} A)\longrightarrow \mathcal{K}(\mathscr{Y}(A),\Bbbk)$ induced by the functor $\mathbf{F}:\mathbf{C}^b(\textup{proj} A)\longrightarrow s(\mathscr{Y}(A),\Bbbk)$, which is given as follows:

\vspace{,5cm}

\emph{In objects}, $\mathbf{P}^{\bullet}\in\mathbf{K}^b(\textup{proj} A)$,
   	$$
   	\widehat{\mathbf{F}}(\mathbf{P}^{\bullet} )={\mathbf{F}}(\mathbf{P}^{\bullet}).
   	$$

\vspace{,3cm}

\emph{In morphisms}, $\overline{\varphi}^{\bullet}$ in $\mathbf{K}^b(\textup{proj} A)$,
   	$$
   	\widehat{\mathbf{F}}(\overline{\varphi}^{\bullet})=\overline{\mathbf{F}(\varphi^{\bullet})},
   	$$
where $\overline{\varphi}^{\bullet}$ and $\overline{\mathbf{F}(\varphi^{\bullet})}$ denote the equivalence classes under the relations ``$\sim$'' and ``$\simeq$'', respectively. 
   	
   	The second consequence is that $\widehat{\mathbf{F}}$ is an embedding functor, because $\ker \widehat{\mathbf{F}} = 0^\bullet$. Here, $\ker\widehat{\mathbf{F}}$ is defined by the complexes $\mathbf{P}^{\bullet}$ such that $\widehat{\mathbf{F}}(\mathbf{P}^{\bullet})=0$ is the zero object, which corresponds to the matrix where all the blocks are empty.

\vspace{,5cm}

Since the categories $\mathbf{K}^b(\textup{proj} A)$ and $\mathcal{K}(\mathscr{Y}(A),\Bbbk)$ are both triangulated (see \cite[Theorem 2.3.1, p. 11]{KZ98} and Theorem~\ref{main 1}), we can now present the main results of this section, which guarantee the existence of triangulated functors. The first result follows directly from Remark \ref{diagramkappa} together with Proposition \ref{standard}.


\begin{theorem}
The embedding functor $\widehat{\mathbf{F}}:\mathbf{K}^b(\textup{proj} A)\longrightarrow \mathcal{K}(\mathscr{Y}(A),\Bbbk)$ sends distinguished triangles of $\mathbf{K}^b(\textup{proj}A)$ to distinguished triangles of $\mathcal{K}(\mathscr{Y}(A),\Bbbk)$.
\end{theorem}


It is well known that, if $A$ is an algebra of finite global dimension, the homotopy category $\mathbf{K}^b(\textup{proj} A)$ and the derived category $\mathbf{D}^b(A)$ are both triangulated categories and equivalents as triangulated categories (see \cite[Proposition 3.5.43, pp. 332-333]{Zim14}). Consequently,


\begin{corollary}
If $A$ is a gentle algebra of finite global dimension, then there exists a triangulated embedding functor from $\mathbf{D}^b(A)$ to $\mathcal{K}(\mathscr{Y}(A),\Bbbk)$. 

\end{corollary}    


\section*{Acknowledgements}
G.C. gratefully acknowledges the hospitality and excellent
working conditions at the Universidade Federal do Amazonas where  this work was completed. G.B. and G.C. was partially supported by the Coordena\c{c}\~{a}o de Aperfei\c coamento de Pessoal de N\'ivel Superior -- Brasil (CAPES) --  Finance Code 001. G.C. has been  supported by Fundação de Amparo à Pesquisa do Estado do Amazonas (FAPEAM). 

\end{document}